\documentclass{amsart}
\usepackage{amssymb,verbatim,cite}

\newtheorem{theorem}{Theorem}
\newtheorem{lemma}{Lemma}
\newtheorem{proposition}{Proposition}
\newtheorem{corollary}{Corollary}

\theoremstyle{definition}
\newtheorem{definition}{Definition}
\newtheorem{example}{Example}

\theoremstyle{remark}
\newtheorem{remark}{Remark}

\numberwithin{equation}{section}
\numberwithin{theorem}{section}
\numberwithin{lemma}{section}
\numberwithin{proposition}{section}
\numberwithin{corollary}{section}
\numberwithin{definition}{section}
\numberwithin{example}{section}
\numberwithin{remark}{section}
\newcommand{\thref}[1]{Theorem~{\rm\ref{#1}}}
\newcommand{\prref}[1]{Proposition~{\rm\ref{#1}}}
\newcommand{\leref}[1]{Lemma~{\rm\ref{#1}}}
\newcommand{\coref}[1]{Corollary~{\rm\ref{#1}}}

\newcommand{\deref}[1]{Definition~{\rm\ref{#1}}}
\newcommand{\exref}[1]{Example~{\rm\ref{#1}}}

\newcommand{\reref}[1]{Remark~{\rm\ref{#1}}}

\newcommand{\seref}[1]{Section~{\rm\ref{#1}}}
\newcommand\lbb[1]{\label{#1}}

\def\as{associative}

\def\psalg{pseudo\-algebra}
\def\psalgs{pseudo\-algebras}

\def\tt{\otimes}                               
\def\bt{\boxtimes}
\def\<{\langle}
\def\>{\rangle}
\def\ti{\widetilde}
\def\wti{\widetilde}
\def\what{\widehat}

\def\d{\partial}

\def\tsum{{\textstyle\sum}}

\def\surjto{\twoheadrightarrow}                
\def\injto{\hookrightarrow}                    
\def\isoto{\xrightarrow{\sim}}                 
\def\st{\; | \;}                               
\def\symm{S}                                   

\def\mmod{\;\mathrm{mod}\;}

\def\di{{\mathrm{d}}}
\def\diz{\di_0}
\def\diru{\di^{\mathrm{R}}}
\def\dizru{\di_0^{\mathrm{R}}}

\newcommand{\kk}{\mathbf{k}}       
\newcommand{\CC}{\mathbb{C}}       
\newcommand{\ZZ}{\mathbb{Z}}       

\newcommand{\fg}{\mathfrak{g}}

\def\al{\alpha}                         
\def\be{\beta}
\def\ga{\gamma}

\def\de{\delta}
\def\De{\Delta}
\def\ep{\varepsilon}
\def\eps{\epsilon}
\def\io{\iota}
\def\la{\lambda}

\def\om{\omega}
\def\Om{\Omega}

\def\si{\sigma}

\def\th{\theta}
\def\Th{\Theta}

\def\g{{\mathfrak{g}}}      
\def\dd{{\mathfrak{d}}}
\def\db{\bar\dd}
\def\cz{{\mathfrak{c}}_0}

\def\k{{\mathfrak{k}}}
\def\gl{{\mathfrak{gl}}}
\def\gld{\gl\,\dd}
\def\gldb{\gl\,\db}
\def\sl{{\mathfrak{sl}}}

\def\sp{{\mathfrak{sp}}}
\def\spd{\sp\,\db}  
\def\csp{{\mathfrak{csp}}}
\def\cspd{\csp\,\db}  

\def\Wd{W(\dd)}    

\def\Kd{K(\dd,\th)}

\def\adsp{\ad^\sp}
\def\A{{\mathcal{A}}}
\def\L{{\mathcal{L}}}
\def\V{{\mathcal{V}}}
\def\W{{\mathcal{W}}}

\def\K{{\mathcal{K}}}

\def\N{\mathcal{N}}

\def\E{\mathcal{E}}

\def\O{\mathcal{O}}

\def\T{\mathcal{T}}           
\def\V{\mathcal{V}}           %
\def\ue{U}                 

\DeclareMathOperator{\sgn}{sgn}
\DeclareMathOperator{\gr}{gr}

\DeclareMathOperator{\Span}{span}

\DeclareMathOperator{\Ind}{Ind}
\DeclareMathOperator{\ad}{ad}

\DeclareMathOperator{\tr}{tr}

\DeclareMathOperator{\sd}{\ltimes}
\DeclareMathOperator{\id}{id}

\DeclareMathOperator{\fil}{F}      

\DeclareMathOperator{\Der}{Der}

\DeclareMathOperator{\Hom}{Hom}
\DeclareMathOperator{\End}{End}

\DeclareMathOperator{\Cur}{Cur}

\DeclareMathOperator{\Rad}{Rad}
\DeclareMathOperator{\Vir}{Vir}
\DeclareMathOperator{\sing}{sing}
\DeclareMathOperator{\coef}{coeff}

\begin{document}

\title[Irreducible Modules over Finite Simple Lie Pseudoalgebras II]
{Irreducible Modules over Finite Simple Lie Pseudoalgebras II. \\
Primitive Pseudoalgebras of Type $K$}

\author[B.~Bakalov]{Bojko Bakalov}
\address{Department of Mathematics,
North Carolina State University, Box 8205,
Raleigh, NC 27695, USA}
\email{bojko\_bakalov@ncsu.edu}
\thanks{The first and third authors were partially supported by NSF grants}

\author[A.~D'Andrea]{Alessandro D'Andrea}
\address{Dipartimento di Matematica,
Istituto ``Guido Castelnuovo'',
Universit\`a di Roma ``La Sapienza'',
00185 Rome, Italy}
\email{dandrea@mat.uniroma1.it}
\thanks{The second author was supported in part by
European Union TMR grant
ERB FMRX-CT97-0100 and CNR grant 203.01.71}

\author[V.~G.~Kac]{Victor G.~Kac}
\address{Department of Mathematics, MIT, Cambridge, MA 02139, USA}
\email{kac@math.mit.edu}

\date\today



\begin{abstract}
One of the algebraic structures that has emerged recently in the study of the operator product expansions of chiral fields in conformal field theory is that of a Lie conformal algebra. A Lie pseudoalgebra is a generalization of the notion of a Lie conformal algebra for which $\CC[\partial]$ is replaced by the universal enveloping algebra $H$ of a finite-dimensional Lie algebra. The finite (i.e., finitely generated over $H$) simple Lie pseudoalgebras were classified in our previous work \cite{BDK}. The present paper is the second in our series on representation theory of simple Lie pseudoalgebras. In the first paper we showed that any finite irreducible module over a simple Lie pseudoalgebra of type $W$ or $S$ is either an irreducible tensor module or the kernel of the differential in a member of the pseudo de Rham complex. In the present paper we establish a similar result for Lie pseudoalgebras of type $K$, with the pseudo de Rham complex replaced by a certain reduction called the contact pseudo de Rham complex. This reduction in the context of contact geometry was discovered by Rumin.
\end{abstract}

\maketitle
\tableofcontents


\section{Introduction}\lbb{sintro}
The present paper is the second in our series of papers on
representation theory of simple Lie pseudoalgebras, the first of
which is \cite{BDK1}.

Recall that a \emph{Lie pseudoalgebra} is a (left) module~$L$ over a
cocommutative Hopf algebra~$H$, endowed with a pseudo-bracket
\begin{displaymath}
  L \otimes L \to (H\otimes H) \otimes_H L \, , \quad
     a \otimes b \mapsto [a*b]\, ,
\end{displaymath}
which is an $H$-bilinear map of $H$-modules, satisfying some
analogs of the skewsymmetry and Jacobi identity of a Lie
algebra bracket (see \cite{BD}, \cite{BDK}).

In the case when~$H$ is the base field~$\kk$, this notion coincides
with that of a Lie algebra.  Any Lie algebra $\fg$ gives rise to
a Lie pseudoalgebra $\Cur \fg = H \otimes \fg$ over~$H$ with
pseudobracket
\begin{equation*}
  [(1 \otimes a) * (1 \otimes b)]= (1 \otimes 1) \otimes_H
       [a,b]\,,
\end{equation*}
extended to the whole $\Cur \fg$ by $H$-bilinearity.

In the case when $H=\kk[\partial]$, the algebra of polynomials in an
indeterminate $\partial$ with the comultiplication $\Delta(\partial)
=\partial \otimes 1 +1\otimes \partial$, the notion of a Lie
pseudoalgebra coincides with that of a \emph{Lie conformal algebra}
\cite{K}. The main result of \cite{DK} states that in this case any finite
(i.e.,~finitely generated over $H=\kk[\partial]$) simple Lie
pseudoalgebra is isomorphic either to $\Cur \fg$ with simple
finite-dimensional~$\fg$, or to the Virasoro pseudoalgebra $\Vir
= \kk [\partial]\ell$, where
\begin{displaymath}
  [\ell * \ell] = (1 \otimes \partial - \partial \otimes 1)
    \otimes_{\kk [\partial]} \ell \, ,
\end{displaymath}
provided that~$\kk$ is algebraically closed of characteristic ~$0$.

In \cite{BDK} we generalized this result to the case when
$H=U(\dd)$, where~$\dd$ is any finite-dimensional Lie algebra.
The generalization of the Virasoro  pseudoalgebra is $W (\dd) = H
\otimes \dd $ with the pseudobracket
\begin{equation*}
  [(1 \otimes a)* (1 \otimes b)] = (1 \otimes 1)\otimes_H (1 \otimes [a,b])
    +  (b\otimes 1)\otimes_H (1\otimes a)- (1 \otimes a)\otimes_H
     (1 \otimes b) \,.
\end{equation*}
The main result of \cite{BDK} is that all nonzero subalgebras of
the Lie pseudoalgebra $W (\dd)$ are simple and non-isomorphic,
and, along with $\Cur \fg$, where $\fg$ is a simple
finite-dimensional Lie algebra, they provide a complete list of
finitely generated over~$H$ simple Lie pseudoalgebras, provided
that~$\kk$ is algebraically closed of characteristic~$0$.
Furthermore, in \cite{BDK} we gave a description of all
subalgebras of $W (\dd)$.  Namely, a complete list consists of
the ``primitive'' series $S (\dd ,\chi)$, $H (\dd ,\chi ,\omega)$
and $K (\dd ,\theta)$, and their ``current'' generalizations.

In \cite{BDK1} we constructed all {\em finite} 
(i.e.,~finitely generated over $H =U (\dd)$) 
irreducible modules over
the Lie pseudoalgebras $W (\dd)$ and $S (\dd,\chi)$.  The
simplest nonzero module over $W (\dd)$ is $\Omega^0 (\dd)
=H$,  given by
\begin{equation}
  \label{eq:1.3}
  (f \otimes a) * g =- (f \otimes ga) \otimes_H 1\, , \quad
     f,g \in H\, , \, a \in \dd \, .
\end{equation}
A generalization of this construction, called a tensor $W
(\dd)$-module, is as follows \cite{BDK1}.  First, given a Lie
algebra~$\fg$, define the semidirect sum $W (\dd) \sd \Cur \fg$ as
a direct sum  as $H$-modules, for which $W (\dd)$ is a subalgebra
and $\Cur \fg$ is an ideal, with the following pseudobracket
between them:
\begin{equation*}
    [(f \otimes a) * (g \otimes b)]=-(f \otimes ga)\otimes_H
       (1 \otimes b) \, ,
\end{equation*}
where $f,g \in H$, $a \in \dd$, $b \in \fg$.  Given a
finite-dimensional $\fg$-module $V_0$, we construct a
representation of $W (\dd) \sd \Cur \fg$ in $V=H \otimes V_0$
by (cf.~\eqref{eq:1.3}):
\begin{equation}
  \label{eq:1.5}
\bigl( (f \otimes a) \oplus (g \otimes b) \bigr) * (h \otimes v) =- (f \otimes ha)
    \otimes_H (1\otimes v) + (g \otimes h) \otimes_H
    (1\otimes  bv)\, ,
\end{equation}
where $f,g,h \in H$, $a \in \dd$, $b \in \fg$, $v \in V_0$.

Next, we define an embedding of $W (\dd)$ in $W (\dd) \sd
\Cur (\dd \oplus \gl\, \dd)$ by
\begin{equation}
  \label{eq:1.6}
    1 \otimes \partial_i \mapsto (1 \otimes \partial_i) \oplus
    \bigl( (1 \otimes \partial_i) \oplus (1 \otimes \ad \partial_i+
      \sum_j \partial_j \otimes e^j_i )\bigr)\, ,
\end{equation}
where $\{ \partial_i \}$ is a basis of $\dd$ and $\{ e^j_i \}$ a
basis of $\gl \,\dd$, defined by $e^j_i (\partial_k)=\de^j_k\partial_i$.
Composing this embedding with the action \eqref{eq:1.5} of $W (\dd)
\sd \Cur \fg$, where $\fg = \dd \oplus \gl\, \dd$, we
obtain a $W (\dd)$-module $V=H \otimes V_0$ for each $(\dd \oplus
\gl\, \dd)$-module~$V_0$.  This module is called a {\em tensor
$W (\dd)$-module} and is denoted $\T (V_0)$.

The main result of \cite{BDK1} states that any finite irreducible
$W (\dd)$-module is a unique quotient of a tensor module $\T
(V_0)$ for some finite-dimensional irreducible $(\dd \oplus \gl
\,\dd)$-module $V_0$, describes all cases when $\T (V_0)$ are not
irreducible, and provides an explicit construction of their
irreducible quotients, called the {\em degenerate $W (\dd)$-modules}.
Namely, we prove in \cite{BDK1} that all degenerate $W
(\dd)$-modules occur as images of the differential~$\di$ in the
$\Pi$-twisted pseudo de Rham complex of $W (\dd)$-modules
\begin{equation}
  \label{eq:1.7}
    0 \to \Omega^0_{\Pi} (\dd)  \xrightarrow{\di} \Omega^1_{\Pi}
       (\dd) \xrightarrow{\di} \cdots \xrightarrow{\di}
       \Omega^{\dim \dd}_{\Pi} (\dd)\, .
\end{equation}
Here $\Pi $ is a finite-dimensional irreducible $\dd$-module and
$\Omega^n_{\Pi}(\dd) = \T (\Pi \otimes \bigwedge^n \dd^*)$ is the space
of pseudo $n$-forms.

In the present paper we construct all finite
irreducible modules over the contact
Lie pseudoalgebra $K (\dd ,\theta)$, where~$\dd$ is a Lie algebra
of odd dimension $2N+1$ and~$\theta$ is a contact linear function
on~$\dd$.  To any $\theta \in \dd^*$ one can associate a
skewsymmetric bilinear form $\omega$ on $\dd$, defined by $\omega
(a \wedge b) = -\theta ([a,b])$.  The linear function~$\theta$ is called
{\em contact} if $\dd$ is a direct sum of subspaces $\db = \ker
\theta$ and $\ker \omega$.  In this case $\dim\ker \omega =1$ and
there exists a unique element $\partial_0\in \ker \omega$ such
that $\theta (\partial_0)=-1$.  Furthermore, the restriction
of~$\omega$ to $\db$ is non-degenerate; hence we can choose
dual bases $\{ \partial_i \}$ and $\{ \partial^i \}$ of $\db$,
i.e.,~$\omega (\partial^i \wedge \partial_j)=\delta^i_j$ for $i,j=1, \ldots
,2N$. Then the element
\begin{equation*}
r=\sum^{2N}_{i=1} \partial_i \otimes \partial^i \in H \otimes H
\end{equation*}
is skewsymmetric and independent of the choice of dual bases.

The Lie pseudoalgebra $K (\dd,\theta)$ is defined as a free
$H$-module~$He$ of rank~$1$ with the following pseudobracket:
\begin{equation*}
    [e * e] = (r + \partial_0 \otimes 1-1 \otimes \partial_0)
       \otimes_H e\, .
\end{equation*}
There is a unique pseudoalgebra embedding of $K (\dd, \theta)$ in
$W (\dd)$, which is given by
\begin{equation*}
     e \mapsto -r + 1 \otimes \partial_0 \, .
\end{equation*}
We will denote again by~$e$ its image in $W (\dd)$.  Let $\sp\,
\dd$ (respectively $\sp\, \db$) be the subalgebra of the Lie algebra $\gl
\,\dd$ (resp.~$\gl\, \db$), consisting of $A \in \gl\, \dd$
(resp.~$A\in \gl \,\db$),
such that $\omega (Au \wedge v) =-\omega (u \wedge Av)$ for all $u,v \in \dd$
(resp.~$\db$).  Let $\csp\,\dd = \sp \,\dd \oplus\kk I'$, where
$I'(\partial_0) = 2\partial_0$, $I'|_{\db} = I_{\db}$, and $\csp\,
\db = \sp\, \db \oplus\kk I_{\db}$.  We have an obvious surjective
Lie algebra homomorphism of the Lie algebra $\sp\, \dd$
onto the (simple) Lie algebra $\sp \,\db$, and of $\csp\, \dd$ onto $\csp\,
\db$.  We show that the image of $e \in W (\dd)$ under the map
\eqref{eq:1.6} lies in $W (\dd) \sd \Cur (\dd \oplus \csp \,\dd)$.
Hence each $(\dd \oplus \csp\, \db)$-module $V_0$, being a $(\dd
\oplus \csp \,\dd)$-module, gives rise to a $K (\dd ,\theta)$-module
$\T (V_0)= H \otimes V_0$, with the action given by \eqref{eq:1.5}.  These
are the {\em tensor modules} $\T(V_0)$ over $K (\dd ,\theta)$.

In the present paper we show that any finite irreducible $K (\dd
,\theta)$-module is a unique quotient of a tensor module
$\T(V_0)$ for some finite-dimensional irreducible $(\dd \oplus \csp
\db)$-module $V_0$. We describe all cases when the $K (\dd
,\theta)$-modules $\T(V_0)$ are not irreducible and give an explicit
construction of their irreducible quotients called degenerate $K
(\dd ,\theta)$-modules.
It turns out that all degenerate $K (\dd ,\theta)$-modules again
appear as images of the differential in a certain complex of $K
(\dd ,\theta)$-modules, which we call the $\Pi$-twisted contact
pseudo de Rham complex, obtained by a certain reduction
of the $\Pi$-twisted pseudo de Rham complex \eqref{eq:1.7}.  The idea of
this reduction is borrowed from Rumin's reduction of the de Rham
complex on a contact manifold \cite{Ru}.

As a corollary of our results we obtain the classification
of all degenerate modules over the contact Lie--Cartan algebra $K_{2N+1}$, along with a
description of the corresponding singular vectors given (without proofs) in
\cite{Ko}. Moreover, we obtain an explicit construction of these
modules.

We will work over an algebraically closed field $\kk$ of characteristic $0$.
Unless otherwise specified, all vector spaces, linear maps and tensor
products will be considered over $\kk$. Throughout the paper, $\dd$ will be a
Lie algebra of odd dimension $2N+1<\infty$.


\section{Preliminaries}\lbb{sprel}

In this section we review some facts and notation that
will be used throughout the paper. 


\subsection{Forms with constant coefficients}\lbb{sfcc}
%
Consider the cohomology complex of the Lie algebra $\dd$
with trivial coefficients:
\begin{equation}\lbb{domcc}
0 \to \Om^0 \xrightarrow{\diz}\Om^1\xrightarrow{\diz} \cdots \xrightarrow{\diz} \Om^{2N+1}
\,, \qquad \dim\dd=2N+1 \,,
\end{equation}
where
$\Om^n = \textstyle\bigwedge^n\dd^*$.
Set
$\Om = \textstyle\bigwedge^\bullet\dd^*
= \textstyle\bigoplus_{n=0}^{2N+1} \Om^n$
and $\Om^n = \{0\}$ if $n<0$ or $n>2N+1$.
We will think of the elements of $\Om^n$ as skew-symmetric
\emph{$n$-forms}, i.e., linear maps from $\bigwedge^n \dd$ to $\kk$.
Then the \emph{differential} $\diz$ is given by the formula
($\al\in\Om^n$, $a_i\in\dd$):
\begin{equation}\lbb{d0al}
\begin{split}
(\diz&\al)(a_1 \wedge \dots \wedge a_{n+1})
\\
&= \sum_{i<j} (-1)^{i+j} \al([a_i,a_j] \wedge a_1 \wedge \dots \wedge
\what a_i \wedge \dots \wedge \what a_j \wedge \dots \wedge a_{n+1})
\end{split}
\end{equation}
if $n\ge1$, and $\diz\al = 0$ if $\al\in\Om^0=\kk$.
Here, as usual, a hat over a term means that it is omitted in the wedge
product.

Recall also that the \emph{wedge product} of two forms
$\al\in\Om^n$ and $\be\in\Om^p$
is defined by:
\begin{equation}\lbb{wedge}
\begin{split}
(\al&\wedge\be)(a_1 \wedge \dots \wedge a_{n+p})
\\
&= \frac1{n!p!} \sum_{\pi\in\symm_{n+p}} (\sgn\pi) \,
\al(a_{\pi(1)} \wedge \dots \wedge a_{\pi(n)}) \,
\be(a_{\pi(n+1)} \wedge \dots \wedge a_{\pi(n+p)}) \,,
\end{split}
\end{equation}
where $\symm_{n+p}$ denotes the symmetric group on $n+p$ letters
and $\sgn\pi$ is the sign of the permutation $\pi$.

The wedge product, defined by \eqref{wedge}, makes $\Om$
an \as\ graded-commutative algebra:
for $\al\in\Om^n$, $\be\in\Om^p$, $\ga\in\Om$, we have
\begin{equation}\lbb{omasgc}
\al\wedge\be = (-1)^{np} \be\wedge\al \in\Om^{n+p} , \qquad
(\al\wedge\be)\wedge\ga = \al\wedge(\be\wedge\ga) \,.
\end{equation}
The differential $\diz$ is an odd derivation of $\Om$\,:
\begin{equation}\lbb{dizder}
\diz(\al\wedge\be) = \diz\al\wedge\be + (-1)^n \al\wedge\diz\be \,.
\end{equation}
For $a\in\dd$, define operators $\io_a\colon\Om^n\to\Om^{n-1}$ by
\begin{equation}\lbb{ioal}
(\io_a\al)(a_1 \wedge \dots \wedge a_{n-1})
=\al(a \wedge a_1 \wedge \dots \wedge a_{n-1}) \,,
\qquad\quad a_i\in\dd \,.
\end{equation}
Then each $\io_a$ is also an odd derivation of $\Om$.
For $A\in\gld$, denote by $A\cdot$ its action on $\Om$\,; explicitly,
\begin{equation}\lbb{acdotal}
(A\cdot\al)(a_1 \wedge \dots \wedge a_n)
= \sum_{i=1}^n\, (-1)^i \al(A a_i \wedge a_1 \wedge \dots \wedge \what a_i
\wedge \dots \wedge a_n) \,.
\end{equation}
Each $A\cdot$  is an even derivation of $\Om$\,:
\begin{equation}\lbb{acder}
A\cdot(\al\wedge\be) = (A\cdot\al)\wedge\be + \al\wedge(A\cdot\be) \,,
\end{equation}
and we have the following Cartan formula
for the coadjoint action of $\dd$\,:
\begin{equation}\lbb{cartan}
(\ad a)\cdot = \diz\io_a + \io_a\diz \,.
\end{equation}
The latter implies that $(\ad a)\cdot$ commutes with $\diz$.

\subsection{Contact forms on $\dd$}\lbb{sliedd}
{}From now on we will assume that the Lie algebra $\dd$ admits a \emph{contact form}
$\th\in\Om^1=\dd^*$, i.e., a $1$-form such that
\begin{equation}\lbb{cont1}
\th \wedge \underbrace{\om\wedge\dots\wedge\om}_{N} \ne 0
\,, \qquad \text{where} \quad \om=\diz\th \,.
\end{equation}
Consider the kernel of $\om$, i.e., the space of all elements
$a\in\dd$ such that $\io_a\om=0$.
Equation \eqref{cont1} implies that
$\ker\om$ is $1$-dimensional and $\th$ does not vanish on it.
We let $s\in\ker\om$ be the unique element for which
$\th(s)=-1$, and let $\db\subset\dd$ be the kernel of $\th$.
Then it is easy to deduce the following lemma (cf.\ \cite{BDK}).

\begin{lemma}\lbb{ldd1}
With the above notation, we have a direct sum of vector subspaces\/
$\dd=\db\oplus\kk s$ such that
\begin{equation}\lbb{cont2}
[a,b] = \om(a \wedge b) s \mod \db \,, \qquad a,b\in\dd \,.
\end{equation}
The restriction of\/ $\om$ to\/ $\db\wedge\db$ is nondegenerate,
$\io_s \om = 0$, and\/ $[s,\db] \subset\db$.
\end{lemma}

Note that not every Lie algebra of odd dimension admits a contact form.
In particular, it is clear from the above lemma that $\dd$ cannot be abelian.
Also, the Lie algebra $\dd$ cannot be simple other than $\sl_2$
(see \cite[Example 8.6]{BDK}).
Here are two examples of pairs $(\dd,\th)$ taken from \cite[Section 8.7]{BDK}.

\begin{example}\lbb{edd1}
Let $\dd=\sl_2$ with the standard basis $\{e,f,h\}$, and let
$\th(h)=1$, $\th(e)=\th(f)=0$. Then
$s=-h$, $\db = \Span\{e,f\}$, and $\om(e \wedge f) = -1$.
\end{example}
\begin{example}\lbb{edd2}
Let $\dd$ be the Heisenberg Lie algebra with a basis $\{a_i,b_i,c\}$
and the only nonzero brackets $[a_i,b_i]=c$ for $1\le i \le N$,
and let $\th(c)=1$, $\th(a_i)=\th(b_i)=0$. Then
$s=-c$, $\db = \Span\{a_i,b_i\}$,
and $\om(a_i \wedge b_i) = -1$.
\end{example}

Let $\bar\om$ be the restriction of $\om$ to $\db\wedge\db$.
Since $\bar\om$ is nondegenerate, it defines a linear isomorphism
$\phi\colon\db\to\db^*$, given by $\phi(a) = \io_a \bar\om$.
The inverse map $\phi^{-1}\colon\db^*\to\db$
gives rise to a skew-symmetric element $r\in\db\tt\db$
such that $\phi^{-1}(\al) = (\al\tt\id)(r)$ for $\al\in\db^*$.
Explicitly, let us choose a basis
$\{\d_0, \d_1, \dots, \d_{2N}\}$ of $\dd$
such that $\d_0=s$ and $\{\d_1, \dots, \d_{2N}\}$
is a basis of $\db$, and let $\{x^0,\dots,x^{2N}\}$ be the dual basis of
$\dd^*$ so that $\<x^j,\d_k\>=\de^j_k$.

We set $\om_{ij} = \om(\d_i \wedge \d_j)$,
and we denote by  $(r^{ij})_{i,j=1,\dots,2N}$
the inverse matrix to $(\om_{ij})_{i,j=1,\dots,2N}$,
so that
\begin{equation}\lbb{cont5}
\sum_{j=1}^{2N} r^{ij} \om_{jk} = \de^i_k \,,
\qquad i,k=1,\dots,2N \,.
\end{equation}
Then
\begin{equation}\lbb{cont6}
r = \sum_{i,j=1}^{2N} r^{ij} \d_i\tt\d_j
= \sum_{i=1}^{2N} \d_i\tt\d^i
= -\sum_{i=1}^{2N} \d^i\tt\d_i \,,
\end{equation}
where
\begin{equation}\lbb{cont7}
\d^i = \sum_{j=1}^{2N} r^{ij} \d_j \,, \qquad
\om(\d^i \wedge \d_k) = \de^i_k
\qquad \text{for} \quad i,k=1,\dots,2N \,.
\end{equation}
We also have
\begin{equation}\lbb{cont8}
\om(\d^i \wedge \d^j)= \< x^i, \d^j \> = - r^{ij} = r^{ji} \,.
\end{equation}

Recall that a basis $\{\d_1, \dots, \d_{2N}\}$ of $\db$
is called \emph{symplectic} iff it satisfies
\begin{equation}\lbb{cont9}
\om(\d_i \wedge \d_{i+N}) = 1 = - \om(\d_{i+N} \wedge \d_i) \,,
\quad \om(\d_i \wedge \d_j) = 0 \quad\text{for}\quad |i-j| \ne N \,.
\end{equation}
In this case we have
\begin{equation}\lbb{cont10}
\d^i = -\d_{i+N} \,, \quad \d^{i+N} = \d_i \,,
\qquad i=1,\dots,N \,,
\end{equation}
which implies that
\begin{equation}\lbb{cont11}
r = \sum_{i=1}^N ( \d_{i+N} \tt \d_i - \d_i \tt \d_{i+N} ) \,.
\end{equation}

Note that, by \eqref{wedge},
\begin{equation}\lbb{thom1}
\th=-x^0 \,, \qquad
\om = \frac12 \sum_{i,j=1}^{2N}\, \om_{ij} x^i \wedge x^j \,,
\end{equation}
and when the basis $\{\d_1, \dots, \d_{2N}\}$ of $\db$ is symplectic,
we have
\begin{equation}\lbb{thom2}
\om = \sum_{i=1}^{N}\, x^i \wedge x^{i+N} \,.
\end{equation}

\subsection{The Lie algebras $\spd$ and $\cspd$}

In this subsection we continue to use the notation from the previous
one. In particular, recall that $\{\d_0, \dots, \d_{2N}\}$
is a basis of $\dd$ and $\{x^0,\dots,x^{2N}\}$ is the dual basis of
$\dd^*$, while restriction to nonzero indices gives dual bases of
$\db$ and $\db^*$.

We will identify $\End\db$ with $\db \tt \db^*$ as a vector
space. In more detail, the elementary matrix $e_i^j \in\End\db$
is identified with the element $\d_i \tt x^j \in \db \tt
\db^*$, where $e_i^j (\d_k) = \de^j_k \d_i$. Notice that $(\d \tt x)
(\d') = \langle x, \d'\rangle \d$, so that the composition $(\d \tt
x) \circ (\d' \tt x')$ equals $\langle x, \d'\rangle \d \tt x'$. We
will adopt a raising index notation for elements of $\End \db$ as
well, so that
\begin{equation}\lbb{eij}
e^{ij} = \d^i \tt x^j = \sum_{k=1}^{2N} r^{ik} e_k^j \,,
\qquad i\ne0 \,.
\end{equation}

\begin{definition}\lbb{dspd}
We denote by $\spd=\sp(\db,\bar\om)$ the Lie algebra of all $A\in\gldb$
such that $A\cdot\bar\om=0$.
\end{definition}

Since the $2$-form $\bar\om$ is nondegenerate, the Lie algebra
$\spd$ is isomorphic to the \emph{symplectic} Lie algebra $\sp_{2N}$,
and in particular it is simple. It will be sometimes convenient
to embed $\spd$ in $\gld$ by identifying $\gldb$ with a subalgebra of $\gld$.
We will also consider the Lie subalgebra $\cspd = \spd \oplus \kk I'$
of $\gld$, where
\begin{equation}\lbb{i'1}
I' = 2e_0^0 + \sum_{i =1}^{2N} e_i^i \in \gld \,.
\end{equation}
Note that $\cspd$
is a trivial extension of $\spd$ by the central ideal $\kk I'$.

\begin{lemma}\lbb{lglom}
We have
\begin{equation}\lbb{glom1}
e_k^j \cdot \th = \de_{k0}\, x^j \,, \qquad
e_0^j \cdot \om = 0 \,, \qquad
e^{ij} \cdot \om = x^i \wedge x^j \,, \quad i\ne0 \,.
\end{equation}
In particular, $A\cdot\th=A\cdot\om=0$ for all $A\in\spd$ and
\begin{equation}\lbb{glom3}
I' \cdot\th = -2\th \,, \qquad
I' \cdot\om = -2\om \,, \qquad
I' \cdot x^i = -x^i \,, \quad i\neq0 \,.
\end{equation}
\end{lemma}
\begin{proof}
One can deduce from \eqref{acdotal} that $e_k^j \cdot x^i = -\de_k^i x^j$.
Then the first two equations in \eqref{glom1}
are immediate from \eqref{thom1} and \eqref{acder}.
To check the third one, we observe that
\begin{equation*}
e_k^j \cdot \om = \sum_{i=1}^{2N} \, \om_{ki} x^i \wedge x^j \,,
\qquad k\ne0
\end{equation*}
and then apply \eqref{eij}.
Finally, \eqref{glom3} can be deduced from \eqref{i'1} and the above formulas.
\end{proof}
\begin{corollary}\lbb{cfij}
The elements
\begin{equation}\lbb{fij}
f^{ij} = - \frac{1}{2} (e^{ij} + e^{ji}) = f^{ji} \,,
\qquad 1\le i\le j\le 2N
\end{equation}
form a basis of\/ $\spd$.
\end{corollary}

Recalling that $\langle x^i, \d^j\rangle = -r^{ij} = r^{ji}$, we find
\begin{equation}\lbb{glproduct}
e^{ij} \circ e^{kl} = r^{kj} e^{il},
\end{equation}
so that
\begin{equation}\lbb{glbracket}
[e^{ij} , e^{kl}] = r^{kj} e^{il} - r^{il} e^{kj}
\end{equation}
and
\begin{equation}\lbb{spbracket}
[f^{ij}, f^{kl}] = \frac{1}{2} \left( r^{ik} f^{jl} + r^{il} f^{jk}
+ r^{jk} f^{il} + r^{jl} f^{ik}\right).
\end{equation}
Let us also introduce the notation
\begin{equation}\lbb{fij2}
f_i^j = \sum_{a=1}^{2N} \omega_{ia} f^{aj}, \qquad\qquad f_{ij} =
\sum_{a,b=1}^{2N} \omega_{ia} \omega_{jb} f^{ab} \,.
\end{equation}

\begin{lemma}\lbb{lsl2}
{\rm(i)}
For every $i=1, \dots, 2N$ the elements
\begin{equation}\lbb{fij3}
h_i=-2f_i^i,\qquad e_i=f_{ii},\qquad f_i=-f^{ii}
\end{equation}
constitute a standard\/ $\sl_2$-triple.

{\rm(ii)}
The element
\begin{equation}\lbb{cassp}
- \sum_{i,j=1}^{2N} f_{ij} f^{ij} \in \ue(\spd)
\end{equation}
equals the Casimir element corresponding to the invariant bilinear form normalized
by the condition that the square length of long roots is~$2$.
\end{lemma}
\begin{proof}
(i) We have:
$$[f^i_i, f^{ii}] = \sum_{a=1}^{2N} [\omega_{ia} f^{ai}, f^{ii}] =
\sum_{a=1}^{2N} \omega_{ia} r^{ai}f^{ii} = f^{ii},$$ and similarly
\begin{align*}
[f^i_i, f_{ii}] & = \sum_{a,b,c=1}^{2N} [\omega_{ia} f^{ai},
\omega_{ib}\omega_{ic} f^{bc}]\\
& = \frac{1}{2} \sum_{a,b,c=1}^{2N} \omega_{ia} \omega_{ib}
\omega_{ic}
(r^{ab} f^{ic} + r^{ac} f^{ib} + r^{ib} f^{ac} + r^{ic} f^{ab})\\
& = \frac{1}{2} \sum_{a,b,c=1}^{2N} (\delta_i^b \omega_{ib}
\omega_{ic}f^{ic} + \delta_i^c \omega_{ic} \omega_{ib} f^{ib} -
\omega_{ib}r^{bi} \omega_{ia} \omega_{ic}f^{ac} - \omega_{ic} r^{ci}
\omega_{ia} \omega_{ib} f^{ab}) = -f_{ii} \,.
\end{align*}
Finally,
\begin{align*}
[f^{ii}, f_{ii}] & = [f^{ii}, \sum_{a,b=1}^{2N} \omega_{ia}
\omega_{ib}
f^{ab}]\\
& = \sum_{a,b=1}^{2N} \omega_{ia} \omega_{ib} (r^{ia} f^{ib} +
r^{ib}
f^{ia})\\
& = - \sum_{a,b=1}^{2N} \omega_{ia} r^{ai} \omega_{ib} f^{ib} +
\omega_{ib} r^{bi} \omega_{ia} f^{ia} = -2f^i_i \,,
\end{align*}
proving part (i).

(ii)
Using \eqref{fij} and \eqref{glproduct}, we compute:
\begin{align*}
f_{ij}f^{kl} &= \sum_{a,b=1}^{2N} \omega_{ia} \omega_{jb} f^{ab} f^{kl} \\
&= \frac{1}{4} \sum_{a,b=1}^{2N} \omega_{ia} \omega_{jb} (r^{kb}
e^{al} + r^{lb} e^{ak} + r^{ka} e^{bl} + r^{la} e^{bk}) \,.
\end{align*}
Since by \eqref{eij}, $\tr e^{ij} = r^{ij} = -r^{ji}$,
we obtain
\begin{equation*}
\tr f_{ij} f^{kl} = \frac{1}{2} \sum_{a,b=1}^{2N} \omega_{ia}
\omega_{jb} (r^{kb} r^{al} + r^{lb} r^{ak}) = - \frac{1}{2}
(\delta_i^l \delta_j^k + \delta_i^k \delta_j^l) \,.
\end{equation*}
The trace form is bilinear, symmetric, invariant under
the adjoint action, and gives square length $2$ for long roots of $\spd$
(see, e.g., \cite[Lecture 16]{FH}).
This proves part~(ii).
\end{proof}

The above lemma turns out to be particularly useful when the basis
$\{\d_i\}$ of $\db$ is symplectic (see \eqref{cont9}).
In this case
one has
\begin{equation}\lbb{fij4}
h_i = e_i^i - e_{N+i}^{N+i} \,, \qquad i = 1, \dots, N \,;
\end{equation}
hence $\{h_i\}_{i=1, \dots, N}$ is a basis for the diagonal
Cartan subalgebra of $\spd$ (cf.\ \cite[Lecture 16]{FH}).

Following the notation of \cite{OV},
we denote by $R(\lambda)$ the irreducible $\spd$-module
with highest weight $\lambda$. Recall that the highest weight of
the vector representation $\db$ is the fundamental weight
$\pi_1$, and that
\begin{equation}\lbb{rpin}
\textstyle\bigwedge\nolimits^n \db
\simeq R(\pi_{n}) \oplus R(\pi_{n-2}) \oplus R(\pi_{n-4}) \oplus\cdots
\,, \qquad 0 \leq n \leq N \,,
\end{equation}
where $\pi_n$ are the fundamental weights and we set
$R(\pi_0) = \kk$, $R(\pi_n) = \{0\}$ if $n<0$ or $n>N$.
The following facts are
standard (see, e.g., \cite{OV}, Reference Chapter, Table~5).

\begin{lemma}\lbb{lspfacts}
With the above notation, we have{\rm:}
\begin{gather*}
R(\pi_n) \tt R(\pi_1) \simeq R(\pi_n + \pi_1)
\oplus R(\pi_{n-1}) \oplus R(\pi_{n+1}) \,, \\
\dim R(\pi_n + \pi_1) > \dim R(\pi_n) \,, \qquad 1 \leq n \leq N \,.
\end{gather*}
Furthermore, the Casimir element \eqref{cassp} acts
on $R(\pi_n)$ as scalar multiplication by $n(2N + 2 - n)/2$.
\end{lemma}

\subsection{Bases and filtrations of\/ $\ue(\dd)$ and\/ $\ue(\dd)^*$}
\lbb{sfilued}
Let $\dd$ be a Lie algebra of dimension $2N+1$
with a basis $\{\d_0, \d_1, \dots, \d_{2N}\}$,
as in \seref{sliedd}.
Then its universal enveloping algebra $H=\ue(\dd)$
has a basis
\begin{equation}\lbb{dpbw}
\d^{(I)} = \d_0^{i_0} \dotsm \d_{2N}^{i_{2N}} / i_0! \dotsm
i_{2N}! \,, \qquad I = (i_0,\dots,i_{2N}) \in\ZZ_+^{2N+1} \,.
\end{equation}
Recall that the coproduct $\De\colon H \to H \tt H$ is a homomorphism
of associative algebras defined by $\De(\d) = \d\tt1 +1\tt\d$ for $\d\in\dd$.
Then it is easy to see that
\begin{equation}\lbb{dedn}
\De(\d^{(I)}) = \sum_{ J+K=I }
\d^{(J)} \tt \d^{(K)}.
\end{equation}

The canonical increasing filtration of $\ue(\dd)$ is given by
\begin{equation}\lbb{filued1}
\fil^p \ue(\dd) = \Span_\kk\{ \d^{(I)} \st |I| \le p \} \,,
\qquad \text{where} \quad |I|=i_0+\cdots+i_{2N} \,,
\end{equation}
and it does not depend on the choice of basis of $\dd$.
This filtration is compatible with the structure of a Hopf
algebra (see, e.g., \cite[Section~2.2]{BDK} for more details).
We have:
$\fil^{-1} H = \{0\}$, $\fil^0 H = \kk$, $\fil^1 H = \kk\oplus\dd$.

It is also convenient to define a
different filtration of $\ue(\dd)$, called the \emph{contact filtration}:
\begin{equation}\lbb{filued2}
\fil'^p \ue(\dd) = \Span_\kk\{\d^{(I)} \st |I|' \le p\} \,,
\qquad  \text{where} \quad |I|' = 2i_0 + i_1 + \cdots + i_{N-1} \,.
\end{equation}
This filtration is also compatible with the Hopf algebra structure
on $\ue(\dd)$, and we have $\fil'^0 H = \kk$, $\fil'^1 H = \kk
\oplus \db$, $\fil'^2 H \supset \kk \oplus \dd = \fil^1 H$. It is
easy to see that the two filtrations of $H$ are equivalent.

The dual $X=H^* := \Hom_\kk(H,\kk)$ is a commutative associative algebra.
Define elements $x_I \in X$ by
$\langle x_I, \d^{(J)} \rangle = \de_I^J$,
where, as usual,
$\de_I^J =1$ if $I=J$ and $\de_I^J =0$ if $I\ne J$.
Then, by \eqref{dedn}, we have $x_J x_K = x_{ J+K }$ and
\begin{equation}\lbb{xi1}
x_I = (x^0)^{i_0} \dotsm (x^{2N})^{i_{2N}} \,,
\qquad I = (i_0,\dots,i_{2N}) \in\ZZ_+^{2N+1} \,,
\end{equation}
where
\begin{equation}\lbb{xi2}
x^i = x_{\ep_i} \,,
\qquad \ep_i = (0,\dots,0,\underset{i}1,0,\dots,0) \,,
\quad i=0,\dots,2N \,.
\end{equation}
Therefore, $X$ can be identified with the algebra $\O_{2N+1} =
\kk[[t^0,t^1,\dots,t^{2N}]]$ of formal power series in $2N+1$
indeterminates.

There are left and right actions of $\dd$ on $X$ by derivations
given by
\begin{align}
\lbb{dx1}
\langle \d x, f\rangle &= -\langle x, \d f\rangle \,,
\\
\lbb{dx2}
\langle x \d, f\rangle &= -\langle x, f\d\rangle \,,
\qquad \d\in\dd, \; x \in X , \; f \in H \,,
\end{align}
where $\d f$ and $f\d$ are products in $H$. These two actions
coincide only when the Lie algebra $\dd$ is abelian. The difference
$\d x - x \d$ gives the coadjoint action of $\d\in\dd$ on $x\in X$.

Let $\fil_p X = (\fil^p H)^\perp$ be the set of elements from
$X=H^*$ that vanish on $\fil^p H$. Then $\{\fil_p X\}$ is a
decreasing filtration of $X$ called the \emph{canonical filtration}.
It has the properties:
\begin{gather}\lbb{dfpx1}
\fil_{-1} X = X \,,  \quad X/\fil_0 X \simeq\kk \,,
\quad \fil_0 X/\fil_1 X \simeq\dd^* \,,
\\
\lbb{dfpx}
(\fil_n X) (\fil_p X) \subset \fil_{n+p+1} X \,,
\quad \dd(\fil_p X) \subset \fil_{p-1} X \,,
\quad (\fil_p X)\dd \subset \fil_{p-1} X \,.
\end{gather}
Note that $\fil_0 X$ is the unique maximal ideal of $X$, and
$\fil_p X = (\fil_0 X)^{p+1}$.
We define a topology of $X$ by considering $\{\fil_p X\}$
as a fundamental system of neighborhoods of $0$.
We will always consider $X$ with this topology, while $H$ and $\dd$ with
the discrete topology. Then $X$ is a linearly compact algebra
(see \cite[Chapter~6]{BDK}),
and the left and right actions of $\dd$ on
it are continuous (see \eqref{dfpx}).

Similar statements hold for the filtration
$\fil'_p X = (\fil'^p H)^\perp$, namely:
\begin{gather}\lbb{dfpxprime}
(\fil'_n X) (\fil'_p X) \subset \fil'_{n+p+1} X \,,
\quad \db(\fil'_p X) \subset \fil'_{p-1} X \,,
\quad (\fil'_p X)\db \subset \fil'_{p-1} X \,,
\\
\lbb{dfpxprime2}
\d_0(\fil'_p X) \subset \fil'_{p-2} X \,,
\quad (\fil'_p X)\d_0 \subset \fil'_{p-2} X \,.
\end{gather}
We will call $\{\fil'_p X\}$ the \emph{contact filtration}.
It is equivalent to the canonical filtration $\{\fil_p X\}$.

We can consider $x^i$ as elements of $\dd^*$; then $\{x^i\}$
is a basis of $\dd^*$ dual to the basis $\{\d_i\}$ of $\dd$,
i.e., $\<x^i,\d_j\>=\de^i_j$. Let $c_{ij}^k$ be the structure constants of
$\dd$ in the basis $\{\d_i\}$, so that $[\d_i,\d_j]=\sum\,c_{ij}^k\d_k$.
Then we have the following formulas for the left and right actions of
$\dd$ on $X$ (see, e.g., \cite[Lemma 2.2]{BDK1}):
\begin{align}
\lbb{dacton1}
\d_i x^j &= -\de_i^j -\sum_{k<i}\, c_{ik}^j x^k \mod\fil_1 X \,,
\\
\lbb{dacton2}
x^j \d_i &= -\de_i^j + \sum_{k>i}\, c_{ik}^j x^k \mod\fil_1 X \,.
\end{align}

\section{Lie Pseudoalgebras and Their Representations}\lbb{sbdef}
In this section we review the definitions and results about
Lie \psalgs\ from \cite{BDK, BDK1}, which will be needed in the paper.

\subsection{Hopf algebra notations}\lbb{shopf}
Let $H$ be a cocommutative Hopf algebra with a coproduct $\De$,
a counit $\ep$, and an antipode $S$.
We will use the following notation (cf.\ \cite{Sw}):
\begin{align}
\lbb{de1}
\De(h) &= h_{(1)} \tt h_{(2)} = h_{(2)} \tt h_{(1)} \,, 
\\
\lbb{de2}
(\De\tt\id)\De(h) &= (\id\tt\De)\De(h) = h_{(1)} \tt h_{(2)} \tt h_{(3)} \,,
\\
\lbb{de3}
(S\tt\id)\De(h) &= h_{(-1)} \tt h_{(2)} \,, \qquad\quad h\in H \,.
\end{align}
Then the axioms of antipode and counit can be written
as follows:
\begin{align}
\lbb{antip}
h_{(-1)} h_{(2)} &= h_{(1)} h_{(-2)} = \ep(h),
\\
\lbb{cou}
\ep(h_{(1)}) h_{(2)} &= h_{(1)} \ep(h_{(2)}) = h,
\end{align}
while the fact that $\De$ is a homomorphism of algebras
translates as:
\begin{equation}
\lbb{deprod}
(fg)_{(1)} \tt (fg)_{(2)} = f_{(1)} g_{(1)} \tt f_{(2)} g_{(2)},
\qquad f,g\in H.
\end{equation}
Eqs.\ \eqref{antip}, \eqref{cou} imply the following
useful relations:
\begin{equation}
\lbb{cou2}
h_{(-1)} h_{(2)} \tt h_{(3)} = 1\tt h
= h_{(1)} h_{(-2)} \tt h_{(3)}.
\end{equation}

The following lemma, which follows from \cite[Lemma 2.3]{BDK},
plays an important role in the paper.

\begin{lemma}\lbb{lhhh}
%
For any $H$-module $V$, the linear maps
\begin{equation*}
H \tt V \to (H \tt H) \tt_H V \,, \qquad
h \tt v \mapsto (h \tt 1) \tt_H v
\end{equation*}
and
\begin{equation*}
H \tt V \to (H \tt H) \tt_H V \,, \qquad
h \tt v \mapsto (1 \tt h) \tt_H v
\end{equation*}
are isomorphisms of vector spaces.
\end{lemma}

The dual $X=H^* := \Hom_\kk(H,\kk)$ becomes a commutative associative algebra
under the product defined by
\begin{equation}
\langle xy, h\rangle = \langle x, h_{(1)}\rangle\langle y h_{(2)}\rangle
\,, \qquad h\in H, \; x,y\in X \,.
\end{equation}
$X$ admits left and right actions of $H$, given by
(cf.\ \eqref{dx1}, \eqref{dx2}):
\begin{align}
\lbb{hx}
\langle hx, f\rangle &= \langle x, S(h)f\rangle \, ,
\\
\langle xh, f\rangle &= \langle x, f S(h)\rangle \,,
\qquad h,f\in H, \; x,y\in X \,.
\end{align}
They have the following properties:
\begin{align}
\lbb{hxy}
h(xy) &= (h_{(1)}x) (h_{(2)}y) \,,
\\
\lbb{xyh}
(xy)h &= (x h_{(1)}) (y h_{(2)}) \,,
\\
\lbb{fxg}
h(xg) &= (hx)g, \qquad h,g\in H, \; x,y \in X.
\end{align}

\subsection{Lie pseudoalgebras and their representations}\lbb{slieps}
Let us recall the definition of a Lie \psalg\ from
\cite[Chapter~3]{BDK}.
A \emph{pseudobracket\/} on a left $H$-module $L$
is an $H$-bilinear map
\begin{equation}\lbb{psprod}
  L \tt L \to (H \tt H) \tt_H L \,, \quad
  a \tt b \mapsto [a * b] \,,
\end{equation}
where we use the
comultiplication $\Delta\colon H \to H \tt H$ to define
$(H\tt H) \tt_H L$.
We extend the pseudobracket \eqref{psprod} to maps
$(H^{\tt 2} \tt_H L) \tt L \to H^{\tt 3}
\tt_H L$ and $L \tt (H^{\tt 2}
\tt_H L) \to H^{\tt 3}\tt_H L$ by letting:
\begin{align}
\lbb{psprod1}
[(h\tt_{H}a)*b] &= \sum \, (h \tt 1)\, (\Delta\tt\id)(g_i)
 \tt_H c_i \,,
\\
\lbb{psprod2}
[a*(h\tt_{H}b)] &= \sum \, (1 \tt h)\, (\id\tt\Delta)(g_i)
 \tt_H c_i \,,
\intertext{where $h\in H^{\tt 2}$, $a,b\in L$, and}
\lbb{psprod3}
[a*b] &= \sum \, g_i \tt_H c_i
\qquad\text{with \; $g_i\in H^{\tt2}$, $c_i \in L$.}
\end{align}

A \emph{Lie \psalg\/} is a left $H$-module equipped with
a pseudobracket satisfying the
following skewsymmetry and Jacobi identity axioms:
\begin{align}
  \lbb{psss}
  [b*a] &= -(\sigma \tt_H \id)\, [a*b] \,, \\
  \lbb{psjac}
  [[a*b]*c] &= [a*[b*c]] - ((\sigma \tt \id)\tt_H\id)\,
  [b*[a*c]] \,.
\end{align}
Here, $\sigma\colon H \tt H \to H \tt H$ is the permutation
of factors, and the compositions $[[a*b]*c]$, $[a*[b*c]]$
are defined using \eqref{psprod1}, \eqref{psprod2}.

The definition of a module over a Lie \psalgs\
is an obvious modification of the above.
A \emph{module\/} over a Lie \psalg\ $L$
is a left $H$-module $V$ together with an $H$-bilinear map
\begin{equation}\lbb{psprod4}
  L \tt V \to (H \tt H) \tt_H V \,, \quad
  a \tt v \mapsto a * v
\end{equation}
that satisfies ($a,b\in L$, $v\in V$)
\begin{equation}\lbb{psrep}
[a*b]*v = a*(b*v) - ((\si\tt\id)\tt_H\id) \, (b*(a*v)) \,.
\end{equation}
An $L$-module $V$ will be called \emph{finite\/}
if it is finitely generated as an $H$-module.

\begin{remark}\lbb{rlv0}
If $V$ is a torsion module over $H$, then the action of $L$ on $V$
is trivial, i.e., $L*V = \{0\}$ (see \cite[Corollary~10.1]{BDK}).
Notice that this holds whenever $V$ is finite dimensional and $H =
\ue(\dd)$ with $\dim \dd > 0$.
\end{remark}


Some of the most important Lie \psalgs\ are described in the
following examples (see \cite{BDK}).

\begin{example}\lbb{ecur}
For a Lie algebra $\g$, the current Lie pseudoalgebra $\Cur\g=H\tt\g$
has an action of $H$ by left multiplication on the first tensor factor
and a pseudobracket
\begin{equation}\lbb{curbr*}
[(f\tt a)*(g\tt b)]
= (f\tt g)\tt_H(1\tt [a,b]) \,.
\end{equation}
\end{example}
\begin{example}\lbb{ewd}
Let $H=\ue(\dd)$ be the universal enveloping algebra of a Lie algebra $\dd$.
Then $\Wd=H\tt\dd$ has the structure
of a Lie pseudoalgebra with the pseudobracket
\begin{equation}\lbb{wdbr*}
\begin{split}
[(f\tt a)*(g\tt b)]
&= (f\tt g)\tt_H(1\tt [a,b])
\\
&- (f\tt ga)\tt_H(1\tt b) + (fb\tt g)\tt_H(1\tt a) \,.
\end{split}
\end{equation}
The formula
\begin{equation}\lbb{wdac*}
(f\tt a)*g = -(f\tt ga)\tt_H 1
\end{equation}
defines the structure of a $\Wd$-module on $H$.
\end{example}
\begin{example}\lbb{ewdcur}
The semidirect sum $\Wd\sd\Cur\g$ contains $\Wd$ and $\Cur\g$
as subalgebras and has the pseudobracket
\begin{equation}\lbb{wdcurg}
[(f\tt a)*(g\tt b)] = -(f\tt ga)\tt_H (1\tt b)
\end{equation}
for $f,g \in H=\ue(\dd)$, $a\in\dd$, $b\in\g$
(cf.\ \eqref{wdac*}).
\end{example}

Let $U$ and $V$ be two $L$-modules. A map
$\be\colon U\to V$ is a \emph{homomorphism\/} of $L$-modules
if $\be$ is $H$-linear and satisfies
\begin{equation}\lbb{psprod6}
\bigl( (\id\tt\id)\tt_H \be \bigr) (a*u)
= a * \be(u) \,, \qquad a\in L \,, \; u\in U \,.
\end{equation}

A subspace $W\subset V$ is an \emph{$L$-submodule\/} if
it is an $H$-submodule and
$L*W \subset (H\tt H)\tt_H W$,
where $L*W$ is the linear span of all elements $a*w$
with $a\in L$ and $w\in W$.
A submodule $W\subset V$ is called \emph{proper\/} if $W\ne V$.
An $L$-module $V$ is \emph{irreducible\/} (or \emph{simple})
if it does not contain any nonzero proper $L$-submodules and
$L*V \ne \{0\}$.

\begin{remark}\lbb{rlmod}
{\rm(i)}
Let $V$ be a module over a Lie \psalg\ $L$ and let
$W$ be an $H$-submodule of $V$. By  \leref{lhhh},
for each $a\in L$, $v\in V$, we can write
\begin{equation}\lbb{lmod1}
a*v = \sum_{I\in\ZZ_+^{2N+1}} (\d^{(I)} \tt 1) \tt_H v'_I \,,
\qquad v'_I \in V \,,
\end{equation}
where the elements $v'_I$ are uniquely determined by $a$ and $v$.
Then $W\subset V$ is an $L$-submodule iff
it has the property that all $v'_I \in W$ whenever $v \in W$.
This follows again from \leref{lhhh}.

{\rm(ii)}
Similarly, for each $a\in L$, $v\in V$, we can write
\begin{equation}\lbb{lmod2}
a*v = \sum_{I\in\ZZ_+^{2N+1}} (1 \tt \d^{(I)}) \tt_H v''_I \,,
\qquad v''_I \in V \,,
\end{equation}
and $W$ is an $L$-submodule iff $v''_I \in W$ whenever $v \in W$.
\end{remark}

\subsection{Twistings of representations}\lbb{stwrep}
Let $L$ be a Lie \psalg\ over $H=\ue(\dd)$, and let
$\Pi$ be any finite-dimensional $\dd$-module.
In \cite[Section 4.2]{BDK1}, we introduced a covariant functor $T_\Pi$
from the category of finite $L$-modules to itself. In the present paper
we will use it only in the special case when all the modules are free
as $H$-modules. For a finite $L$-module $V=H\tt V_0$, which is free over $H$,
we choose a $\kk$-basis $\{ v_i \}$ of $V_0$, and write the action of $L$
on $V$ in the form
\begin{equation}\lbb{twrep1}
a*(1\tt v_i) = \sum_j\, (f_{ij} \tt g_{ij}) \tt_H (1 \tt v_j)
\end{equation}
where $a \in L$, $f_{ij}, g_{ij} \in H$.

\begin{definition}\lbb{dtwrep}
The \emph{twisting} of $V$ by $\Pi$ is the $L$-module
$T_\Pi(V) = H\tt \Pi\tt V_0$, where $H$ acts by a left multiplication on
the first factor and
\begin{equation}\lbb{twrep3}
a*(1\tt u\tt v_i) =
\sum_j\, \bigl( f_{ij} \tt {g_{ij}}_{(1)} \bigr)
\tt_H \bigl( 1 \tt {g_{ij}}_{(-2)} u \tt v_j \bigr)
\end{equation}
for $a\in L$, $u\in\Pi$.
\end{definition}

The facts that $T_\Pi(V)$ is an $L$-module and that
the action of $L$ on it is independent of the choice of basis of $V_0$ follow
from \cite[Proposition 4.2]{BDK1}.
Let us now recall how $T_\Pi$ is defined on homomorphisms of
$L$-modules. Consider two finite $L$-modules, $V=H\tt V_0$ and $V'=H\tt V'_0$.
Choose $\kk$-bases $\{ v_i \}$ and $\{ v'_i \}$ of
$V_0$ and $V'_0$, respectively.
For a homomorphism of $L$-modules $\be\colon V\to V'$, write
\begin{equation}\lbb{twrep6}
\be(1\tt v_i) = \sum_j\, h_{ij} \tt v'_j \,,
\qquad h_{ij} \in H \,.
\end{equation}
Then
$T_\Pi(\be)\colon T_\Pi(V)\to T_\Pi(V')$ is given by
\begin{equation}\lbb{twrep8}
T_\Pi(\be)(1\tt u\tt v_i) =
\sum_j\, {h_{ij}}_{(1)} \tt {h_{ij}}_{(-2)} u \tt v'_j  \,.
\end{equation}
Thanks to \cite[Proposition 4.3]{BDK1}, the map $T_\Pi(\be)$ is a homomorphism of $L$-modules,
independent of the choice of bases.

Note that $T_\Pi$ can be defined on the category of (free) $H$-modules.
The next result concerns only the $H$-module structure.

\begin{proposition}\lbb{ptwrep}
{\rm(i)}
The functor $T_\Pi$ is exact on free $H$-modules, i.e., if\/
$V \xrightarrow{\be} V' \xrightarrow{\be'} V''$
is a short exact sequence of finite free $H$-modules,
then the sequence\/
$T_\Pi(V) \xrightarrow{T_\Pi(\be)} T_\Pi(V')
\xrightarrow{T_\Pi(\be')} T_\Pi(V'')$
is exact.

\medskip
{\rm(ii)}
Let\/ $\be\colon V\to V'$ be a homomorphism between two free $H$-modules.
If the image of\/ $\be$ has a finite codimension
over\/ $\kk$, then the image of\/ $T_\Pi(\be)$
has a finite codimension in\/ $T_\Pi(V')$.
\end{proposition}
\begin{proof}
Consider the linear map
\begin{equation*}
F\colon H\tt\Pi \to H\tt\Pi \,,
\qquad h\tt u \mapsto h_{(1)} \tt h_{(-2)} u \,,
\end{equation*}
which was introduced in the proof of \cite[Lemma 5.2]{BDK1}.
{}From \eqref{cou2} it is easy to see that $F$ is a linear isomorphism and
\begin{equation*}
F^{-1}(h\tt u) = h_{(1)} \tt h_{(2)} u \,,
\qquad h \in H , \; u \in \Pi \,.
\end{equation*}
Since $F$ is a linear isomorphism, both statements of the proposition are true if and
only if they are true for $(F^{-1}\tt\id) T_\Pi(\be)$
instead of $T_\Pi(\be)$. In this case, they follow easily from the identity
\begin{equation*}
(F^{-1}\tt\id) T_\Pi(\be)(1\tt u\tt v_i) =
\sum_j\, h_{ij} \tt u \tt v'_j
= \si_{12} \bigl( u \tt \be(1\tt v_i) \bigr) \,,
\end{equation*}
where $\si_{12}$ is the transposition of the first and second factors.
\end{proof}

\subsection{Annihilation algebras of Lie pseudoalgebras}\lbb{spsanih}
For a Lie \psalg\ $L$, we set $\A(L)=X\tt_H L$,
where as before $X=H^*$, and
we define a Lie bracket on $\L=\A(L)$ by the formula
(cf.\ \cite[Eq.~(7.2)]{BDK}):
\begin{equation}\lbb{alliebr}
[x\tt_H a, y\tt_H b] = \sum\, (x f_i)(y g_i) \tt_H c_i \,,
\quad\text{if}\quad
[a*b] = \sum\, (f_i\tt g_i) \tt_H c_i \,.
\end{equation}
Then $\L$ is a Lie algebra, called the \emph{annihilation algebra\/}
of $L$ (see \cite[Section~7.1]{BDK}).
We define a left action of
$H$ on $\L$ in the obvious way:
\begin{equation}\lbb{hactsonl}
h(x\tt_H a) = hx\tt_H a.
\end{equation}
In the case $H=\ue(\dd)$, the Lie algebra $\dd$ acts on $\L$ by
derivations. The semidirect sum $\ti\L = \dd\sd\L$ is called
the \emph{extended annihilation algebra}.

Similarly, if $V$ is a module over a Lie \psalg\ $L$, we let $\A(V)=X\tt_H V$,
and we define an action of $\L=\A(L)$ on $\A(V)$ by:
\begin{equation}\lbb{alav}
(x\tt_H a)(y\tt_H v) = \sum\, (x f_i)(y g_i) \tt_H v_i \,,
\quad\text{if}\quad
a*v = \sum\, (f_i\tt g_i) \tt_H v_i \,.
\end{equation}
We also define an $H$-action on $\A(V)$ similarly to \eqref{hactsonl}.
Then $\A(V)$ is an $\ti\L$-module \cite[Proposition~7.1]{BDK}.

When $L$ is a finite $H$-module, we can define a filtration on
$\L$ as follows (see \cite[Section~7.4]{BDK} for more details).
We fix a finite-dimensional vector
subspace $L_0$ of $L$ such that $L = HL_0$, and set
\begin{equation}\lbb{fill}
\fil_p\L = \Span_\kk \{ x \tt_H a \in\L \st x\in\fil_p X \,, \; a\in L_0 \}
\,, \qquad p\ge -1 \,.
\end{equation}
The subspaces $\fil_p \L$ constitute a decreasing
filtration of $\L$, satisfying
\begin{equation}\lbb{filbr}
[\fil_n \L, \fil_p \L] \subset \fil_{n+p-\ell} \L \,,
\quad \dd (\fil_p \L) \subset \fil_{p-1} \L \,,
\end{equation}
where $\ell$ is an integer depending only on the choice of $L_0$.
Notice that the filtration just defined depends on the choice of
$L_0$, but the topology that it induces does not \cite[Lemma~7.2]{BDK}.
We set $\L_p = \fil_{p+\ell} \L$, so that $[\L_n, \L_p] \subset \L_{n+p}$.
In particular, $\L_0$ is a subalgebra of $\L$.

We also define a filtration of $\ti\L$ by letting
$\fil_{-1}\ti\L=\ti\L$, $\fil_p\ti\L = \fil_p\L$ for $p\ge0$,
and we set $\ti\L_p = \fil_{p+\ell} \ti\L$.
An $\ti\L$-module $V$ is called \emph{conformal\/}
if every $v\in V$ is killed by some $\L_p$; in other words, if $V$ is
a continuous $\ti\L$-module when endowed with the discrete topology.


The next two results from \cite{BDK} play a crucial role in our
study of representations (see \cite{BDK}, Propositions 9.1 and 14.2,
and Lemma 14.4).

\begin{proposition}\lbb{preplal2}
Any module $V$ over the Lie \psalg\ $L$ has a natural structure of a
conformal $\ti\L$-module, given by the action of\/ $\dd$ on $V$ and by
\begin{equation}\lbb{axm2}
(x\tt_H a) \cdot v
= \sum\, \< x f_i, {g_i}_{(1)} \> \, {g_i}_{(2)} v_i \,,
\quad\text{if}\quad
a*v = \sum\, (f_i\tt g_i) \tt_H v_i
\end{equation}
for $a\in L$, $x\in X$, $v\in V$.

Conversely, any conformal $\ti\L$-module $V$
has a natural structure of an $L$-module, given by
\begin{equation}\lbb{prpl2}
a*v = \sum_{I\in\ZZ_+^{2N+1}} \bigl( S(\d^{(I)}) \tt1 \bigr)\tt_H
\bigl( (x_I\tt_H a) \cdot v \bigr) \,.
\end{equation}

Moreover, $V$ is irreducible
as an $L$-module iff it is irreducible as an $\ti\L$-module.
\end{proposition}

\begin{lemma}\lbb{lkey2}
Let $L$ be a finite Lie \psalg\
and $V$ be a finite $L$-module. For $p\ge -1-\ell$, let
\begin{equation*}\lbb{kernv}
\ker_p V
= \{ v \in V \st \L_p \, v = 0 \},
\end{equation*}
so that, for example, $\ker_{-1-\ell} V = \ker V$
and\/ $V = \bigcup \ker_p V$.
Then all vector spaces\/
$\ker_p V / \ker V$ are finite dimensional.
In particular, if\/ $\ker V=\{0\}$, then
every vector $v\in V$ is contained in a finite-dimensional subspace
invariant under~$\L_0$.
\end{lemma}

\section{Primitive Lie Pseudoalgebras of Type $K$}
\lbb{sprim}
Here we introduce the main objects of our study: the
Lie pseudoalgebra $\Kd$ and its
annihilation algebra $\K$ (see \cite[Chapter~8]{BDK}).
We will review the (unique) embedding of $\Kd$ into $\Wd$
and the induced embedding of annihilation algebras.
Throughout this section, $\dd$ will be a Lie algebra of odd dimension
$2N+1$, and $\th\in\dd^*$ will be a contact form, as in \seref{sliedd}.
As before, let $H=\ue(\dd)$.

\subsection{Definition of $\Kd$}\lbb{subprim}
Recall the elements $r\in\dd\tt\dd$ and $s\in\dd$ introduced
in \seref{sliedd} and notice that $r$ is skew-symmetric.
It was shown in \cite[Lemma 8.7]{BDK} that $r$ and $s$ satisfy
the following equations:
\begin{align}
\lbb{cybe3}
[r,\De(s)] &= 0 \,,
\\
\lbb{cybe4}
([r_{12}, r_{13}] + r_{12} s_3) + \text{{\rm{cyclic}}} &= 0 \,,
\end{align}
where we use the standard notation $r_{12}=r\tt1$, $s_3=1\tt1\tt s$, etc.,
and ``cyclic'' denotes terms obtained by
applying the two nontrivial cyclic permutations.

\begin{definition}\lbb{dkd}
The Lie \psalg\ $\Kd$ is defined as a free $H$-module of rank one, $He$,
with the following pseudobracket
\begin{equation}\lbb{kd1}
[e*e]=(r+s\tt1-1\tt s)\tt_H e \,.
\end{equation}
\end{definition}

The fact that $\Kd$ is a Lie \psalg\ follows from \eqref{cybe3}, \eqref{cybe4};
see \cite[Section 4.3]{BDK}.
By \cite[Lemma 8.3]{BDK}, there is an injective homomorphism of Lie \psalgs
\begin{equation}\lbb{kd2}
\io\colon \Kd\to\Wd \,, \qquad e \mapsto -r + 1 \tt s \,,
\end{equation}
where $\Wd=H\tt\dd$ is from \exref{ewd}.
Moreover, this is the unique nontrivial
homomorphism from $\Kd$ to $\Wd$ \cite[Theorem 13.7]{BDK}.
{}From now on, we will often identify $\Kd$ with its image in $\Wd$
and will write simply $e$ instead of $\io(e)$.
In the notation of \seref{sliedd}, we have the formula
\begin{equation}\lbb{ioe2}
e = 1 \tt \d_0 - \sum_{i=1}^{2N} \d_i \tt \d^i
\,.
\end{equation}

\subsection{Annihilation algebra of $\Wd$}\lbb{subanw}
Let $\W = \A(\Wd)$ be the annihilation algebra of the Lie \psalg\
$\Wd$ (see \seref{spsanih}).
Since $\Wd=H \tt \dd$, we have $\W = X \tt_H (H \tt \dd)
\simeq X \tt \dd$, so we can identify $\W$ with $X\tt\dd$. Then
the Lie bracket in $\W$ becomes ($x,y\in X$, $a,b\in\dd$):
\begin{equation}\lbb{Wbra}
[x \tt a, y \tt b] = xy \tt [a, b] - x(ya) \tt b + (xb)y \tt a \,,
\end{equation}
while the left action of $H$ on $\W$ is given by: $h(x \tt a) = hx
\tt a$. The Lie algebra $\dd$ acts on $\W$ by derivations. We
denote by $\ti\W$ the extended annihilation algebra
$\dd\sd\W$, where
\begin{equation}\lbb{dactw}
[\d, x \tt a] = \d x \tt a \,, \qquad \d,a\in\dd, \; x\in X \,.
\end{equation}


We choose $L_0=\kk\tt\dd$ as a subspace of $\Wd$ such that
$\Wd=H L_0$, and we obtain the following filtration of $\W$:
\begin{equation}\lbb{wp}
\W_p = \fil_p \W = \fil_p X\tt_H L_0 \equiv \fil_p X\tt\dd \,,
\qquad p \ge -1 \,.
\end{equation}
This is a decreasing filtration of $\W$, satisfying $\W_{-1}=\W$
and $[\W_i, \W_j] \subset \W_{i+j}$. Note that $\W/\W_0 \simeq
\kk\tt\dd \simeq \dd$ and $\W_0/\W_1 \simeq \dd^*\tt\dd$.

\begin{lemma}[\!\cite{BDK1}]\lbb{cwbra}
For $x\in \fil_0 X$, $a\in\dd$, the map
\begin{equation}
(x\tt a) \mod \W_1 \mapsto -a \tt (x \mmod \fil_1 X)
\end{equation}
is a Lie algebra isomorphism from
$\W_0/\W_1$ to\/ $\dd\tt\dd^*\simeq\gld$.
Under this isomorphism, the adjoint action of\/ $\W_0/\W_1$ on
$\W/\W_0$ coincides with the standard action of\/
$\gld$ on~$\dd$.
\end{lemma}

The action \eqref{wdac*} of $\Wd$ on $H$ induces a corresponding action of the
annihilation algebra $\W$ on $\A(H)\equiv X$:
\begin{equation}
(x\tt a) y = -x (ya), \qquad x,y \in X, \; a\in \dd.
\end{equation}
Since $\dd$ acts on $X$ by continuous derivations, the Lie algebra
$\W$ acts on $X$ by continuous derivations.  The isomorphism $X
\simeq \O_{2N+1}$ from \seref{sfilued} induces a Lie algebra
homomorphism $\W \to W_{2N+1}= \Der \O_{2N+1}$.
In fact, this is an isomorphism compatible with the filtrations
\cite[Proposition~3.1]{BDK1}.
Recall that the canonical filtration of the Lie--Cartan algebra $W_{2N+1}$ is given by
\begin{equation}\lbb{filpwn2}
\fil_p W_{2N+1} = \Bigl\{ \sum_{i=0}^{2N} f_i \frac\d{\d t^i}
\; \Big| \; f_i \in \fil_p \O_{2N+1} \Bigr\} \,,
\end{equation}
where $\fil_p \O_{2N+1}$ is the $(p+1)$-st power of the maximal
ideal $(t^0,\dots,t^{2N})$ of $\O_{2N+1}$.

The Euler vector field
\begin{equation}\lbb{euler}
E := \sum_{i=0}^{2N} t^i \frac{\d}{\d t^i} \in \fil_0 W_{2N+1}
\end{equation}
gives rise to a grading of $\O_{2N+1}$ and a grading
$W_{2N+1;j}$ ($j \geq -1$) of $W_{2N+1}$ such that
\begin{equation}\lbb{filpwn3}
\fil_p W_{2N+1} = \prod_{j\geq p} W_{2N+1;j}\,,
\quad \fil_p W_{2N+1} / \fil_{p+1} W_{2N+1} \simeq W_{2N+1;p}\,.
\end{equation}

We define the \emph{contact filtration} of $\W$ by (see \eqref{filued2}):
\begin{equation}\lbb{wp1}
\W'_p = \fil'_p \W
= (\fil'_p X \tt \db) \oplus (\fil'_{p+1} X \tt \kk s) \,.
\end{equation}
Introduce the \emph{contact Euler vector field}
\begin{equation}\lbb{Keuler}
E' := 2 t^0 \frac{\d}{\d t^0} + \sum_{i=1}^{2N} t^i \frac{\d}{\d t^i}
\in \fil_0 W_{2N+1} \cap \fil'_0 W_{2N+1} \,.
\end{equation}
Then the adjoint action of $E'$ decomposes $W_{2N+1}$ as a direct
product of eigenspaces $W'_{2N+1;j}$ ($j \geq -1$), on which
$\ad E'$ acts as multiplication by $j$. One defines
\begin{equation}\lbb{filpwn}
\fil'_p W_{2N+1} = \prod_{j\geq p} W'_{2N+1;j}
\end{equation}
so that
\begin{equation}
\quad \fil'_p W_{2N+1} / \fil'_{p+1} W_{2N+1} \simeq W'_{2N+1;p} \,.
\end{equation}
The filtration $\{\fil'_p W_{2N+1}\}$ induces on $W_{2N+1}$ the same
topology as the filtration $\{\fil_p W_{2N+1}\}$.

\subsection{Annihilation algebra of $\Kd$}\lbb{subans}



We define a filtration on the annihilation algebra
$\K=\A(\Kd)$ by
\begin{equation}\lbb{kpprime}
\K'_p = \fil'_p \K = \fil'_{p+1} X \tt_H e
\,, \qquad p \ge -2 \,.
\end{equation}
This filtration is equivalent to the one defined in
\seref{spsanih} by choosing $L_0 = \kk e$, because
the filtrations $\{\fil'_p X\}$ and $\{\fil_p X\}$
are equivalent.

Recall that the canonical injection $\io$ of the subalgebra $\Kd$ in
$\Wd$ induces an injective Lie algebra homomorphism
$\A(\io) \colon \K \to \W$ that allows
us to view $\K$ as a subalgebra of $\W$.
In more detail, by \eqref{ioe2} we have
\begin{equation}\lbb{ioe3}
\A(\io)(x \tt_H e) = x \tt \d_0 - \sum_{i=1}^{2N} x \d_i \tt \d^i
\,, \qquad x \in X \,.
\end{equation}

\begin{lemma}\lbb{lprimefil}
The contact filtrations of\/ $\K$ and $\W$ are compatible, i.e., one
has $\K'_p = \K \cap \W'_p$. In particular, $[\K'_m, \K'_n] \subset
\K'_{m+n}$.
\end{lemma}
\begin{proof}
Any element of $\K'_p$ has the form $x \tt_H e$ with $x \in
\fil'_{p+1} X$. Then, by \eqref{ioe3}, \eqref{wp1} and \eqref{dfpxprime},
its image in $\W$ lies in $\W'_p$. Therefore,
$\K'_p \subset \K \cap \W'_p$. The opposite inclusion is proved similarly.
\end{proof}

Composing the isomorphism $\W \to W_{2N+1}$ with the injection $\K
\to \W$, one obtains a map $\phi\colon \K \to W_{2N+1}$, whose image
however does not coincide with $K_{2N+1} \subset W_{2N+1}$. Recall
that $K_{2N+1}$ is the Lie subalgebra of $W_{2N+1}$ consisting of
vector fields preserving the standard contact form $\di t^0 +
\sum_{i=1}^N t^i \di t^{N+i}$ up to multiplication by a function,
i.e., by an element of $\O_{2N+1}$ (see \cite[Chapter 6]{BDK} and
the references therein).

\begin{proposition}\lbb{plck}
There exists a ring automorphism $\psi$ of\/ $\O_{2N+1}$ such that
the induced Lie algebra automorphism $\psi$ of\/ $W_{2N+1}$ satisfies
$\phi(\K) = \psi(K_{2N+1})$.
\end{proposition}
\begin{proof}
The proof is similar to that of \cite[Proposition 3.6]{BDK1}. The
image $\phi(\K)$ is the Lie algebra of all vector fields
preserving a certain contact form up to multiplication
by an element of $\O_{2N+1}$ \cite[Proposition 8.3]{BDK1}.
We can find a change of variables conjugating this contact form to
the standard contact form $\di t^0 + \sum_{i=1}^N t^i \di t^{N+i}$.
Hence, there exists an automorphism $\psi$ of $\O_{2N+1}$ such
that $\phi(\K) = \psi(K_{2N+1})$.
\end{proof}

We will denote by $\E'$ the lifting to $\K$ of the contact Euler
vector field $E' \in K_{2N+1}$, that is $\E' = \phi^{-1} \psi(E')$.

\begin{remark}
The adjoint action of $\E'$ on $\K$ is semisimple, as it translates
the semisimple action of $E'$ on $K_{2N+1}$. As the automorphism
$\psi$ can be chosen so that the induced homomorphism on the
associated graded Lie algebra equals the identity, one can easily
show that the adjoint action of $\E'$ on $\K$ preserves each $\K'_n$
and that it equals multiplication by $n$ on $\K'_n/\K'_{n+1}$.
\end{remark}

\subsection{The normalizer $\N_{\K}$}\lbb{snk}

It is well known that all derivations of the Lie--Cartan algebras of type
$W$ are inner. This fact was used in \cite[Section 3.3]{BDK1}
to prove that the centralizer of $\W$ in $\ti \W$
consists of elements $\ti\d$ $(\d\in\dd)$
so that the map $\d\mapsto\ti\d$ is an isomorphism of Lie algebras.
We have
\begin{equation}\lbb{tid1}
\ti \d = \d + 1 \tt \d - \ad \d \mod \W_1, \qquad \d \in \dd \,,
\end{equation}
where $\ad \d$ is understood as an element of $\gl\,\dd \simeq
\W_0/\W_1$.

\begin{proposition}\lbb{pknorm}
Elements\/ $\ti \d$ span a Lie subalgebra $\ti \dd \subset
\ti \K$ isomorphic to\/ $\dd$.
The normalizer $\N_{\K}$ of\/ $\K'_p$ in $\ti\K$ coincides with
$\ti\dd \oplus \K'_0$ and is independent of\/ $p\geq 0$.
There is a decomposition as a direct sum of subspaces
$\ti \K = \dd \oplus \N_{\K}$.
\end{proposition}
\begin{proof}
Since all derivations of $\K \simeq K_{2N+1}$ are inner,
there exist elements $\what\d \in \ti\K$ centralizing $\K$
and such that $\what\d = \d \mod \K$ for $\d\in\dd$.
Then $\what\d-\ti\d \in\W$ centralizes $\K$, which implies
$\what\d=\ti\d$, because the centralizer of $\K$ in $\W$ is zero.
Therefore, the centralizer of $\K$ in $\ti\K$
coincides with the centralizer $\ti\dd$ of $\W$ in $\ti\W$.
The other statements follow as in
\cite[Proposition~3.3]{BDK1}.
\end{proof}

The above proposition implies that for every $\d\in\dd$
the element
$\ti \d - \d \in \W$ lies in the subalgebra $\K$,
and hence it can be expressed as a Fourier coefficient $x \tt_H e$
for suitable $x \in X$. In order to do so, let us compute
the images of the first few Fourier coefficients of $e$ under the
identification of $\K$ as a subalgebra of $\W$.

\begin{lemma}\lbb{lfourierk}
The embedding $\A(\iota) \colon \K\to\W$
identifies the following elements{\rm:}
\begin{align}
\tag{i}
1 \tt_H e &\mapsto 1 \tt \d_0 \,;
\\ \tag{ii}
x^j \tt_H e &\mapsto 1 \tt \d^j + x^j \tt \d_0 -
\sum_{0<i<k} c_{ik}^j x^k \tt \d^i \mod \W'_1 \cap \W_1 \,;
\\ \tag{iii}
x^0 \tt_H e &\mapsto x^0 \tt \d_0 - \sum_{0<i<k} \om_{ik} x^k \tt \d^i
 \mod \W'_1 \cap \W_1 \,;
\\ \tag{iv}
x^i x^j \tt_H e &\mapsto 2 f^{ij} \mod \W_1 \,, \qquad i, j \neq 0 \,;
\\ \tag{v}
x^0 x^j \tt_H e &\mapsto x^0 \tt \d^j \mod \W'_1 \cap
\W_1 \,, \qquad j \neq 0 \,;
\\ \tag{vi}
x^i x^j x^k \tt_H e &\mapsto 0 \mod \W'_1 \cap \W_1 \,,
\qquad i, j, k \neq 0 \,.
\end{align}
\end{lemma}
\begin{proof}
The proof is straightforward, using \eqref{ioe3}, \eqref{dacton2},
and \eqref{cont2}.
Note that elements $f^{ij}\in \gld$, defined in \eqref{fij}, need to be
understood by means of the identification $\gld = \W_0/\W_1$ given
in \leref{cwbra}.
\end{proof}
Notice that $\K$ (respectively $\K'_0$, $\K'_1$) is spanned over $\kk$ by
elements (i)-(vi) (resp. (iii)-(vi), (v)-(vi)) modulo $\K'_2$. Also,
$\K'_2 \subset \W'_2 \subset \W_1$, by \leref{lprimefil} and
$\fil'_2 X \subset \fil_1 X$, which follows from $\fil^1 H \subset
\fil'^2 H$.

In the proof of next proposition, we will use the following abelian
Lie subalgebra of $\gld$:
\begin{equation}\lbb{cz}
\cz= x^0 \tt \db = \Span\{ e^{i0} \}_{1\le i\le 2N} = \Span\{ e_i^0
\}_{1\le i\le 2N} \subset\gld \,.
\end{equation}
Note that the semidirect sum $\cz\rtimes\cspd\subset\gld$ is a Lie
algebra containing $\cz$ as an abelian ideal.

\begin{proposition}\lbb{pk0k1}
We have
$\K'_0/\K'_1 \simeq \spd \oplus \kk I' = \cspd$.
\end{proposition}
\begin{proof}
Since elements (iii)-(vi) in the previous lemma all lie in
$\W_0$, and $\K'_2 \subset \W_1$, we have $\K'_0 \subset \W_0$.
Moreover $\W_1 \subset \W_0$ is an ideal, so the inclusion $\K \to
\W$ induces a well-defined Lie algebra homomorphism $\pi\colon \K'_0
\to \W_0/\W_1 \simeq \gl\,\dd$. Observe now that in $\W_0/\W_1$ one
has
\begin{equation}\lbb{Iprime}
\begin{split}
-I' & = 2 x^0 \tt \d_0 + \sum_{i=1}^{2N} x^i \tt \d_i \\
&= 2 x^0 \tt \d_0 + \sum_{i,j=1}^{2N} \omega_{ij} x^i \tt \d^j \\
&= 2 x^0 \tt_H e + 2 \sum_{0<i<j} \omega_{ij} f^{ij} \mod \W_1 \,.
\end{split}
\end{equation}
As a consequence, $I'\in \gl\,\dd$ lies in the image of $\pi$. By
\leref{lfourierk}, $\pi$ is injective on the linear span of elements
(iii)-(v). The image of $\pi$ equals $\cz\rtimes\cspd$, and $\pi$
maps the ideal $\K'_1\subset \K'_0$ onto the ideal $\cz \subset
\cz\rtimes\cspd$, so that $\pi$ induces an isomorphism between
$\K'_0/\K'_1$ and $\cspd$.
\end{proof}

\begin{corollary}\lbb{clieinsp}
Elements $\ti \d\in \ti\K$ satisfy the following $(j\neq0)${}$:$
\begin{align}\lbb{tilded0}
\ti \d_0 - \d_0 & = 1 \tt_H e - \ad \d_0 \mod \K'_1 \,, \\
\lbb{tildedi} \ti \d^j - \d^j & = x^j \tt_H e - \Bigl( \ad \d^j +
x^j \tt \d_0 - \sum_{0<i<k} c_{ik}^j x^k \tt \d^i \Bigr) \mod \K'_1
\,.
\end{align}
\end{corollary}
\begin{proof}
Follows from \eqref{tid1}, \leref{lfourierk} and Propositions
\ref{pknorm} and \ref{pk0k1}.
\end{proof}

The above two statements imply:

\begin{corollary}\lbb{clieinsp2}
Elements
\begin{equation}\lbb{adsp1}
\ad \d_0 \,, \quad \ad \d^j - e^j_0 + \sum_{0<i<k} c_{ik}^j e^{ik}
\,, \quad j\ne0
\end{equation}
lie in $\spd$.
\end{corollary}
\begin{proof}
Indeed, they must lie in $\cspd$ but the
matrix coefficient multiplying $e^0_0$ is zero in both cases.
\end{proof}


Similarly to \cite{BDK,BDK1},
we will say that an $\N_\K$-module $V$ is \emph{conformal}
if $\K'_p$ acts trivially on it for some $p\geq 1$.

\begin{proposition}\lbb{pnkconformal}
The subalgebra $\K'_1 \subset \N_\K$ acts trivially on any irreducible
finite-dimensional conformal $\N_\K$-module.
Irreducible finite-dimensional conformal
$\N_{\K}$-modules are in one-to-one correspondence with irreducible
finite-dimensional modules over the Lie algebra
$\N_{\K}/\K'_1 \simeq \dd \oplus \cspd$.
\end{proposition}
\begin{proof}
The proof is the same as in \cite[Proposition~3.4]{BDK1}. Let $V$ be
a finite-dimensional irreducible conformal $\N_{\K}$-module; then it
is an irreducible module over the finite-dimensional Lie algebra $\g
= \N_\K/\K'_p = \ti \dd \oplus (\K'_0/\K'_p)$ for some $p\geq 1$. We
apply \cite[Lemma~3.4]{BDK1} for $I=\K'_1/\K'_p$ and $\g_0 = (\kk
\E' + \K'_p)/\K'_p$. Note that by \leref{lprimefil}, one has $I
\subset \Rad \g$, and $[\E', \K'_p] \subset \K'_p$. Moreover, the
adjoint action of $\E'$ on $\K'_1$ is invertible. Thus, the
adjoint action of $\E'$ is injective on $I$, and $I$ acts trivially
on $V$. We can then take $p=1$, in which case $\g = \ti \dd
\oplus(\K'_0/\K'_1) \simeq \dd \oplus \cspd$.
\end{proof}

\section{Singular Vectors and Tensor Modules}\lbb{skten}
We start this section by recalling an important class of modules over
the Lie \psalg\ $\Wd$ called tensor modules. Restricting such modules to
$\Kd$ leads us to the definition of a tensor module over $\Kd$.
By investigating singular vectors, we show that every irreducible module
is a homomorphic image of a tensor module. We continue to use the notation
of \seref{sprel}.

\subsection{Tensor modules for $\Wd$}\lbb{stwd}
Consider a Lie algebra $\fg$ with a finite-dimensional representation $V_0$.
Then the semidirect sum Lie \psalg\ $\Wd\sd\Cur\fg$ from \exref{ewdcur}
acts on the free $H$-module $V=H\tt V_0$ as follows
(see \cite[Remark 4.3]{BDK1}):
\begin{equation}\lbb{wcact}
\bigl( (f\tt a) \oplus (g\tt b) \bigr)*(h\tt u)
= -(f\tt ha)\tt_H (1\tt u)
+ (g\tt h)\tt_H (1\tt bu) \,,
\end{equation}
where $f,g,h \in H=\ue(\dd)$, $a\in\dd$, $b\in\g$, $u\in V_0$.
This combines the usual action of $\Cur\g$ on $V$ with the $\Wd$-action
on $H$ given by \eqref{wdac*}.

By \cite[Remark 4.6]{BDK1}, there is an embedding of Lie \psalgs\
$\Wd \injto \Wd\sd\Cur(\dd\oplus\gld)$ given by
\begin{equation}\lbb{wdgld2}
1\tt\d_i \mapsto (1\tt\d_i) \oplus
\bigl( (1\tt\d_i) \oplus (1\tt\ad\d_i + \sum_{j}\, \d_j \tt e_i^j )\bigr)
\,.
\end{equation}
Composing this embedding with the above action \eqref{wcact} for
$\g=\dd\oplus\gld$, we obtain a $\Wd$-module $V=H\tt V_0$
for every $(\dd\oplus\gld)$-module $V_0$. This module $V$ is
called a \emph{tensor module} and denoted $\T(V_0)$. The action
of $\Wd$ on $\T(V_0)$ is given explicitly by \cite[Eq. (4.30)]{BDK1},
which we reproduce here for convenience:
\begin{equation}\lbb{wdgcd3}
\begin{split}
(1\tt \d_i)*(1\tt u) &= (1 \tt 1) \tt_H (1 \tt (\ad\d_i)u)
+ \sum_{j}\, (\d_j \tt 1) \tt_H (1 \tt e_i^j u)
\\
&- (1 \tt \d_i) \tt_H (1 \tt u)
+ (1 \tt 1) \tt_H (1 \tt \d_i u) \,.
\end{split}
\end{equation}

If $\Pi$ is a finite-dimensional $\dd$-module
and $V_0$ is a finite-dimensional $\gld$-module, then
their exterior tensor product $\Pi \bt V_0$ is defined as
the $(\dd\oplus\gld)$-module $\Pi \tt V_0$, where $\dd$
acts on the first factor and $\gld$ acts on the second one.
Following \cite{BDK1}, in this case the tensor module
$\T(\Pi \bt V_0)$ will also be denoted as $\T(\Pi,V_0)$.
Then
\begin{equation}\lbb{tpiv0}
\T(\Pi,V_0) = T_\Pi(\T(\kk,V_0)) \,,
\end{equation}
where $T_\Pi$ is the twisting functor from \deref{dtwrep}.

\subsection{Tensor modules for $\Kd$}\lbb{stkd}
We will identify $\Kd$ with a subalgebra of $\Wd$ via
embedding \eqref{kd2}. Then $\Kd=He$ where $e\in\Wd$ is given by
\eqref{ioe2}. Introduce the $H$-linear map $\tau\colon\Wd\to\Cur\gld$
given by (cf.\ \eqref{wdgld2})
\begin{equation}\lbb{wdgld3}
\tau(h\tt\d_i) = h\tt\ad\d_i + \sum_{j=0}^{2N}\, h\d_j \tt e_i^j \,,
\qquad h\in H \,.
\end{equation}
Then the image of $e$ under the map \eqref{wdgld2} has the form
$e \oplus (e \oplus \tau(e))$.

\begin{definition}\lbb{dadsp}
We define a linear map $\adsp\colon\dd\to\spd$ by $\adsp\d_0=\ad\d_0$
and
\begin{equation}\lbb{adsp2}
\adsp\d^k = \ad \d^k - e^k_0 + \frac12\sum_{i,j=1}^{2N} c_{ij}^k e^{ij} \,,
\qquad k\ne0 \,.
\end{equation}
\end{definition}
\begin{remark}\lbb{radsp}
The fact that the image of $\adsp$ is inside $\spd$ follows from
\coref{clieinsp2} (cf.\ \eqref{fij}, \eqref{adsp1}). One can
show that $\adsp\d^k$ is obtained from $\ad\d^k$ by first restricting
it to $\db\subset\dd$ and then projecting onto $\spd$. This implies that
the map $\adsp$ does not depend on the choice of basis.
\end{remark}
\begin{lemma}\lbb{leee}
With the above notation, we have
\begin{equation}\lbb{eee1}
\tau(e) = (\id\tt\adsp)(e) + \frac12\d_0 \tt I'
- \sum_{i=1}^{2N} \d_i\d_0 \tt e^{i0}
+ \sum_{i,j=1}^{2N} \d_i\d_j \tt f^{ij} \,.
\end{equation}
\end{lemma}
\begin{proof}
Using \eqref{fij} and the $H$-linearity of $\tau$, we find for $i\ne0$
\begin{align*}
\tau(\d_i\tt\d^i)
&= \d_i\tt\ad\d^i + \sum_{j=0}^{2N}\, \d_i\d_j \tt e^{ij}
\\
&= \d_i\tt\ad\d^i + \d_i\d_0 \tt e^{i0}
- \sum_{j=1}^{2N}\, \d_i\d_j \tt f^{ij}
+ \frac12 \sum_{j=1}^{2N}\, [\d_i,\d_j] \tt e^{ij} \,.
\end{align*}
By \eqref{cont2} and \eqref{i'1}, we have
\begin{equation}\lbb{didj}
[\d_i,\d_j] = \om_{ij} \d_0 + \sum_{k=1}^{2N}\, c_{ij}^k \d_k \,,
\qquad i,j\ne0
\end{equation}
and
\begin{equation*}
\sum_{i,j=1}^{2N}\, \om_{ij} \d_0 \tt e^{ij}
= - \sum_{j=1}^{2N}\, \d_0 \tt e_j^j = \d_0 \tt 2(e_0^0-I') \,.
\end{equation*}
The rest of the proof is straightforward.
\end{proof}

Recall the definition of the abelian subalgebra $\cz \subset \gld$
given in \eqref{cz}.

\begin{corollary}\lbb{ceee}
With the above notation, we have{\rm:}
$\tau(e) \in \Cur(\cz\rtimes\cspd)$.
\end{corollary}

Therefore, the image of $e$ under map \eqref{wdgld2}
lies in $\Wd\sd\Cur\g$ where $\g:=\dd\oplus(\cz\rtimes\cspd)$.
Hence, every finite-dimensional $\g$-module $V_0$ gives rise to a $\Kd$-module
$H\tt V_0$ with an action given by \eqref{wcact}. An important special case
is when $\cz$ acts trivially on $V_0$. Since $\cz$ is an ideal in $\g$,
having such a representation is equivalent to having a representation of
the Lie algebra $\dd\oplus\cspd \simeq \g/\cz$.

\begin{definition}\lbb{dtmodk}
{\rm(i)} Let $V_0$ be a finite-dimensional representation of $\dd\oplus\cspd$.
Then the above $\Kd$-module $H\tt V_0$ is called a \emph{tensor module}
and will be denoted as~$\T(V_0)$.

{\rm(ii)} Let $V_0 = \Pi\boxtimes U$, where $\Pi$ is a
finite-dimensional $\dd$-module and $U$ is a finite-dimensional
$\cspd$-module. Then the module $\T(V_0)$ will also be denoted
as~$\T(\Pi,U)$.

{\rm(iii)} Let $V_0$ be as in part (ii), and assume $I'\in \cspd$
acts on $U$ as multiplication by a scalar $c\in \kk$.
Then the module $\T(V_0)$ will
also be denoted as~$\T(\Pi,U,c)$, and similarly the $\cspd$-module
structure on $U$ will be denoted $(U, c)$.
\end{definition}

The action of $e\in\Kd$ on a tensor module $\T(V_0)=H\tt V_0$
is given explicitly by (cf.\ \eqref{ioe2}, \eqref{wcact}, \eqref{eee1}):
\begin{equation}\lbb{ktm}
\begin{split}
e * (1 &\tt u) = -e \tt_H (1 \tt u)
+ (1 \tt 1) \tt_H \bigl(1 \tt (\d_0 + \ad \d_0) u \bigr)
\\
&- \sum_{k=1}^{2N}  \, (\d_k \tt 1) \tt_H
\bigl(1 \tt (\d^k + \adsp \d^k) u \bigr)
+ \frac12 (\d_0 \tt 1) \tt_H (1 \tt I' u)
\\
& + \sum_{i,j=1}^{2N} \, (\d_i \d_j \tt 1) \tt_H (1 \tt f^{ij} u)
\,,\qquad u \in V_0\,.
\end{split}
\end{equation}

\begin{remark}\lbb{rtm}
More generally, if $\cz$ does not act trivially on $V_0$, the above action
\eqref{ktm} is modified by adding the term
\begin{equation*}
- \sum_{i=1}^{2N} (\d_i\d_0 \tt 1) \tt_H (1 \tt e^{i0} u)
\end{equation*}
to the right-hand side (cf.\ \leref{leee}).
\end{remark}

As in \cite{BDK1}, in the sequel it will be convenient to modify the above
definition of tensor module. Let $R$ be a finite-dimensional $(\dd\oplus\cspd)$-module, with an
action denoted as $\rho_R$. We equip $R$ with the following modified action
of $\dd\oplus\cspd$ (cf.\ \cite[Eqs. (6.7), (6.8)]{BDK1}):
\begin{equation}\lbb{wsing2}
\begin{split}
\d u &= (\rho_R(\d) + \tr(\ad\d)) u \,,
\qquad \d\in\dd \,, \; u\in R \,,
\\ 
A u &= (\rho_R(A) - \tr A) u \,,
\qquad A\in\cspd \,, \; u\in R \,.
\end{split}
\end{equation}
Note that, in fact, $\tr A = 0$ for $A\in\spd$ and $\tr I' = 2N+2$.

\begin{definition}\lbb{dvmodk}
Let $R$ be a finite-dimensional $(\dd\oplus\cspd)$-module
with an action $\rho_R$. Then the tensor module $\T(R)$, where $R$
is considered with the modified action \eqref{wsing2}, will be denoted
as~$\V(R)$. As in \deref{dtmodk}, we will also use the
notation $\V(\Pi,U)$ and $\V(\Pi,U,c)$ when $R = \Pi\boxtimes U$ and
$I'$ acts on $U$ as multiplication by a scalar $c$.
\end{definition}

The above definition can be made more explicit as follows:
\begin{equation}\lbb{wsing6}
\begin{split}
\V(\Pi,U,c) &= \T(\Pi\tt\kk_{\tr\ad} , U, c-2N-2 ) \,,
\\
\T(\Pi,U,c) &= \V(\Pi\tt\kk_{-\!\tr\ad} , U, c+2N+2 ) \,,
\end{split}
\end{equation}
where for a trace form $\chi$ on $\dd$ we denote by $\kk_\chi$
the corresponding $1$-dimensional $\dd$-module.

\begin{remark}\lbb{rtm2} (cf.\ \cite[Remark 6.2]{BDK1}).
Let $R$ be a finite-dimensional representation of $\dd\oplus\cspd$, or
more generally, of $\dd\oplus(\cz\rtimes\cspd)$. Using the map $\pi$
from the proof of \prref{pk0k1}, whose image is $\cz\rtimes\cspd$,
we endow $R$ with an action of $\N_{\K}= \ti\dd \oplus \K'_0$.
Moreover, $\cz$ acts trivially on $R$ if and only if $\K'_1$ does.
Then Propositions \ref{preplal2}, \ref{pknorm} and \ref{pk0k1} imply that,
as a $\wti\K$-module, the tensor module $\V(R)$ is isomorphic to the
induced module $\Ind_{\N_{\K}}^{\wti\K} R$.
\end{remark}

The action of $\Kd$ on $\V(R)$ can be derived from \eqref{ktm} and
\eqref{wsing2}. We will need the following explicit form of this
action.

\begin{proposition}\lbb{pktm}
The action of\/ $\Kd$ on a tensor module $\V(R)$ is given by{\rm:}
\begin{equation}\lbb{ksing1}
\begin{split}
e * (1 &\tt u) =
(1 \tt 1) \tt_H \bigl( 1 \tt \rho_R (\d_0 + \ad \d_0) u - \d_0 \tt u \bigr)
\\
&- \sum_{k =1}^{2N}  \, (\d_k \tt 1) \tt_H \bigl( 1 \tt \rho_R
(\d^k + \adsp \d^k) u - \d^k \tt u \bigr)
\\
&+ \frac12 (\d_0 \tt 1) \tt_H \bigl(1 \tt \rho_R(I')u\bigr)
+ \sum_{i,j=1}^{2N} \, (\d_i \d_j \tt 1) \tt_H
\bigl(1 \tt \rho_R (f^{ij})u\bigr) \,.
\end{split}
\end{equation}
\end{proposition}
\begin{proof}
Let us compare \eqref{ksing1} to \eqref{ktm}, using \eqref{wsing2} and the
fact that
\begin{align*}
-(1 &\tt 1) \tt_H (\d_0 \tt u)
+ \sum_{i =1}^{2N}  \, (\d_i \tt 1) \tt_H (\d^i \tt u)
\\
&= -e \tt_H (1 \tt u) - (\d_0 \tt 1) \tt_H (1 \tt u)
+ \sum_{i =1}^{2N}  \, (\d_i \d^i \tt 1) \tt_H (1 \tt u) \,.
\end{align*}
Noting that $\ad\d_0\in\spd$ and $\tr\ad\d_0 = 0$, we see that \eqref{ksing1}
reduces to the following identity
\begin{equation*}
\sum_{i =1}^{2N}\, \d_i \d^i = -N \d_0
- \sum_{k =1}^{2N}\, (\tr\ad\d^k) \d_k \,.
\end{equation*}
By \eqref{cont2}--\eqref{cont7}, we have:
\begin{equation}\lbb{p24}
\begin{split}
2 \sum_{i =1}^{2N}\, \d_i \d^i
&= \sum_{i=1}^{2N}\, [\d_i,\d^i]
= \sum_{i,j=1}^{2N}\, r^{ij} [\d_i,\d_j]
\\
&= \sum_{i,j=1}^{2N}\, r^{ij} \om_{ij} \d_0
+ \sum_{i,j,k=1}^{2N}\, r^{ij} c_{ij}^k \d_k \,,
\end{split}
\end{equation}
and the coefficient of $\d_0$ in the right-hand side is indeed $-2N$.
On the other hand, for $k\ne0$ the fact that $\adsp\d^k\in\spd$ implies
\begin{equation*}
0 = \tr\adsp \d^k
= \tr\ad \d^k + \frac12 \sum_{i,j=1}^{2N}\, r^{ij} c_{ij}^k \,,
\end{equation*}
using that $\tr e^{ij} = r^{ij}$. This completes the proof.
\end{proof}

\begin{remark}\lbb{rktm}
Computing directly
\begin{equation*}
\tr\ad\d^k = \sum_{i=1}^{2N}\, r^{ki} \tr\ad\d_i
= \sum_{i,j=1}^{2N}\, r^{ki} c_{ij}^j \,,
\end{equation*}
we obtain the identities
\begin{equation*}
\sum_{i,j=1}^{2N}\, r^{ki} c_{ij}^j
+ \frac12 \sum_{i,j=1}^{2N}\, r^{ij} c_{ij}^k = 0
\,, \qquad k\ne0 \,.
\end{equation*}
\end{remark}

\subsection{Singular vectors}

The annihilation algebra $\K$ of $\Kd$ has a decreasing filtration
$\{\K'_p\}_{p\geq -2}$ (see \eqref{kpprime}).
For a $\K$-module $V$, we denote by $\ker_p
V$ the set of all $v \in V$ that are killed by $\K'_p$. A
$\K$-module $V$ is called \textit{conformal} iff $V = \bigcup
\ker_p V$. For any $p \geq 0$ the normalizer of $\K'_p$ in
$\ti \K$ is equal to $\N_\K$ due to \prref{pknorm}.
Therefore, each $\ker_p V$ is an $\N_\K$-module, and in fact,
$\ker_p V$ is a representation of the
finite-dimensional Lie algebra $\N_\K/\K'_p =
\ti \dd \oplus (\K'_0 / \K'_p)$. In particular, by \prref{pk0k1},
$\N_\K/\K'_1$ is
isomorphic to the direct sum of Lie algebras $\dd \oplus \cspd$.

Equivalence of the filtrations $\{\K_p\}$ and $\{\K'_p\}$, along
with \prref{preplal2}, implies that any $\Kd$-module has a natural
structure of a conformal $\ti\K$-module and vice versa.

\begin{definition}\lbb{dksing}
For any $\Kd$-module $V$, a \emph{singular vector\/} is an element
$v\in V$ such that $\K'_1\cdot v = 0$. The space of singular
vectors in $V$ will be denoted by $\sing V$. We will denote by
$\rho_{\sing} \colon \dd\oplus\cspd \to \gl(\sing V)$ the
representation obtained from the $\N_\K$-action on $\sing
V\equiv\ker_1 V$ via the isomorphism $\N_\K / \K'_1 \simeq
\dd\oplus\cspd$.
\end{definition}
It follows that a vector $v\in V$ is singular if and only if
\begin{equation}\lbb{vsingf1}
e*v \in (\fil'^2 H \tt \kk) \tt_H V \,,
\end{equation}
or equivalently
\begin{equation}\lbb{vsingf2}
e*v \in (\kk \tt \fil'^2 H) \tt_H V \,.
\end{equation}

\begin{proposition}\lbb{pksing1}
For any nonzero finite $\Kd$-module $V$, the vector space
$\sing V$ is nonzero and the space $\sing V/\ker V$ is finite
dimensional.
\end{proposition}
\begin{proof}
Finite dimensionality of $\ker_p V/\ker V$ for all $p$ follows
from \leref{lkey2}.
To prove that $\sing V \ne \{0\}$, we may assume without loss of
generality that $\ker V = \{0\}$. Since the $\K$-module $V$ is
conformal, $\ker_p V$ is nonzero for some $p\geq 0$. Note
that $\ker_p V$ is preserved by the
normalizer $\N_\K$. Choose an irreducible $\N_\K$-submodule $U
\subset \ker_p V$. As $U$ is finite dimensional,
\prref{pnkconformal} shows that the action of $\K'_1$ on $U$ is
trivial, hence $U \subset \sing V$.
\end{proof}

Note that, by definition,
\begin{equation}\lbb{rhosing1}
\rho_{\sing}(\d) v = \ti\d \cdot v \,, \qquad \d\in\dd \,, \; v\in
\sing V \,,
\end{equation}
and, due to \leref{lfourierk}(iv),
\begin{equation}\lbb{rhosing2}
\rho_{\sing}(f^{ij}) v =
\frac{1}{2}(x^i x^j \tt_H e) \cdot v\,, \qquad v\in
\sing V \,.
\end{equation}
The next result describes the action of $\Kd$ on a singular vector.
It can be derived from \reref{rtm2}, but for completeness we give
a direct proof.

\begin{proposition}\lbb{lksing}
Let\/ $V$ be a $\Kd$-module and\/ $v \in V$ be a singular vector.
Then the action of\/ $\Kd$ on $v$ is given by
\begin{equation}\lbb{Kactionsing}
\begin{split}
e * v &= \sum_{i,j=1}^{2N} \, (\d_i \d_j \tt 1) \tt_H
\rho_{\sing}(f^{ij})v
+ \frac{1}{2} (\d_0 \tt 1) \tt_H \rho_{\sing}(I')v
\\
&- \sum_{k=1}^{2N} \, (\d_k \tt 1) \tt_H \bigl( \rho_{\sing} (\d^k +
\adsp \d^k) v - \d^k v \bigr)
\\
&+ (1 \tt 1) \tt_H \bigl( \rho_{\sing}(\d_0 + \ad \d_0) v - \d_0 v
\bigr) \,.
\end{split}
\end{equation}
\end{proposition}
\begin{proof}
As $\K'_1$ acts trivially on a singular vector $v$, \prref{preplal2}
implies that
\begin{equation}\lbb{preKactionsing}
\begin{split}
e * v = & \sum_{0<i<j}
(S(\d_i \d_j) \tt 1) \tt_H (x^i x^j \tt_H e) \cdot v\\
&+ \frac{1}{2} \sum_{i=1}^{2N}  (S(\d_i^2) \tt 1) \tt_H
( (x^i)^2 \tt_H e ) \cdot v \\
&+ \sum_{k=0}^{2N}  (S(\d_k) \tt 1) \tt_H (x^k \tt_H e) \cdot v\\
& +  (1 \tt 1) \tt_H (1 \tt_H e) \cdot v \,.
\end{split}
\end{equation}
On the other hand, by \coref{clieinsp2} and \leref{lfourierk}(iv), we have
for $k\neq0$:
\begin{align}\lbb{rhosing3}
(1 \tt_H e) \cdot v &
= \ti \d_0 \cdot v - \d_0 v + \rho_{\sing}( \ad \d_0) v  \,,
\\ \lbb{rhosing4}
(x^k \tt_H e) \cdot v & = \ti \d^k \cdot v - \d^k v
+ \rho_{\sing} \Bigl( \ad
\d^k - e^k_0 + \sum_{0<i<j} c_{ij}^k e^{ij} \Bigr) v \,.
\end{align}

Now we rewrite the first summand on the right-hand side of
\eqref{preKactionsing} using that
\begin{equation*}
S(\d_i \d_j) = \d_j \d_i
= \frac12 (\d_i \d_j + \d_j \d_i) - \frac12 [\d_i, \d_j] \,.
\end{equation*}
Then, thanks to \eqref{rhosing2}, the first two summands become
\begin{equation*}
\sum_{i,j =1}^{2N} (\d_i\d_j \tt 1) \tt_H \rho_{\sing}(f^{ij})v
- \sum_{0<i<j} ([\d_i, \d_j] \tt 1) \tt_H
\rho_{\sing}(f^{ij})v \,.
\end{equation*}
This shows that the first summand in \eqref{Kactionsing} matches
with \eqref{preKactionsing}. By \eqref{rhosing1}, \eqref{rhosing3},
the last summands in \eqref{Kactionsing} and \eqref{preKactionsing}
are also equal.

It remains to rewrite
\begin{equation*}
\sum_{0<i<j} ([\d_i, \d_j] \tt 1) \tt_H
\rho_{\sing}(f^{ij})v
+ \sum_{k=0}^{2N}  (\d_k \tt 1) \tt_H (x^k \tt_H e) \cdot v
\end{equation*}
so that it matches the negative of the second and third terms
in the right-hand side of \eqref{Kactionsing}.
Recalling the commutation relations \eqref{didj}, we obtain
\begin{align*}
& (\d_0 \tt 1) \tt_H \Bigl( (x^0 \tt_H e) \cdot v +
\rho_{\sing}\bigl(\sum_{0<i<j} \omega_{ij}
f^{ij}\bigr)v \Bigr)\\
& + \sum_{k =1}^{2N} (\d_k \tt 1) \tt_H \Bigl( (x^k \tt_H e) \cdot v +
\rho_{\sing} \bigl(\sum_{0<i<j} c_{ij}^k
f^{ij}\bigr) v \Bigr) \,.
\end{align*}
By \eqref{Iprime}, the first summand is equal to
$-\frac{1}{2}(\d_0 \tt 1) \tt_H \rho_{\sing}(I') v$.

Finally, by \eqref{rhosing4}, \eqref{fij} and \eqref{adsp2}, we have
\begin{align*}
(x^k & \tt_H e) \cdot v +
\rho_{\sing} \bigr(\sum_{0<i<j} c_{ij}^k
f^{ij}\bigl) v
\\
&= \ti\d^k \cdot v - \d^k v
+ \rho_{\sing} \Bigl( \ad \d^k - e_0^k + \sum_{0<i<j}
c_{ij}^k e^{ij} +
\sum_{0<i<j} c_{ij}^k f^{ij} \Bigr) v
\\
&= \ti\d^k \cdot v - \d^k v  + \rho_{\sing}(\adsp \d^k)v \,.
\end{align*}
This completes the proof.
\end{proof}

\begin{corollary}\lbb{cksing}
Let $V$ be a $\Kd$-module and let $R$ be a nonzero
$(\dd\oplus\cspd)$-submodule of\/ $\sing V$. Denote by $HR$ the
$H$-submodule of\/ $V$ generated by $R$. Then $HR$ is a
$\Kd$-submodule of\/ $V$. In particular, if\/ $V$ is irreducible,
then $V=HR$.
\end{corollary}
\begin{proof}
By \eqref{Kactionsing}, $\Kd * R \subset (H \tt H) \tt_H HR$, and
by $H$-bilinearity, $\Kd * HR \subset (H \tt H) \tt_H HR$.
\end{proof}

\begin{corollary}\lbb{cksing2}
Let\/ $R$ be a finite-dimensional $(\dd\oplus\cspd)$-module with an
action $\rho_R$. Then for the tensor $\Kd$-module $\V(R)=H\tt R$,
we have $\kk \tt R \subset \sing \V(R)$ and
\begin{equation}\lbb{rhosingA}
\rho_{\sing}(A)(1 \tt u) = 1 \tt \rho_R(A)u \,, \qquad A \in \dd
\oplus \cspd, \;\; u \in R \,.
\end{equation}
\end{corollary}

We will call elements of $\kk \tt R \subset \V(R)$ \emph{constant} vectors.
Combining the above results, we obtain the following theorem.

\begin{theorem}\lbb{irrfacttens}
Let\/ $V$ be an irreducible finite $\Kd$-module, and let\/ $R$ be an
irreducible $(\dd \oplus \cspd)$-submodule of\/ $\sing V$. Then $V$ is
a homomorphic image of\/ $\V(R)$. In particular, every irreducible
finite $\Kd$-module is a quotient of a tensor module.
\end{theorem}
\begin{proof}
Comparing \eqref{Kactionsing} and \eqref{ksing1}, we see that
the canonical projection
$\V(R) = H \tt R \to HR$ is a homomorphism of $\Kd$-modules.
However, $HR=V$ by \coref{cksing}.
\end{proof}

We will now show that reducibility of a tensor module depends on
the existence of nonconstant singular vectors.

\begin{definition}
An element $v$ of a $\Kd$-module $V$ is called {\em
homogeneous} if it is an eigenvector for the action of $\E' \in\K$.
\end{definition}

\begin{remark}
Note that the homogeneous components of a singular vector are
still singular, so that a classification of singular vectors will
follow from a description of homogeneous ones.
\end{remark}

\begin{lemma}\lbb{noconstants}
Let\/ $R$ be an irreducible representation of\/ $\dd \oplus \cspd$.
Then any nonzero proper $\Kd$-submodule $M$ of\/ $\V(R)$
does not contain nonzero constant vectors, i.e.,
$M \cap (\kk\tt R) = \{0\}$.
\end{lemma}
\begin{proof}
Both $M$ and $\kk\tt R\subset \sing \V(R)$ are $\N_{\K}$-stable, and the
same is true of their intersection $M_0$.
Since $\K_1'$ acts trivially on $M_0$, it is a representation of
$\N_{\K}/\K_1' \simeq \dd \oplus \cspd$.
The claim now follows from the
irreducibility of $\kk\tt R \simeq R$.
\end{proof}

\begin{corollary}\lbb{constirr}
If\/ $\sing \V(R) = \kk\tt R$, then the $\Kd$-module $\V(R)$ is
irreducible.
\end{corollary}
\begin{proof}
Assume there is a nonzero proper submodule $M$. Then $M$ must contain
some nonzero singular vector. However, $M \cap \sing \V(R) =
\{0\}$ by \leref{noconstants}.
\end{proof}

\begin{proposition}\lbb{constiff}
Every nonconstant homogeneous singular vector in $\V(R)$ is
contained in a nonzero proper submodule. In particular, $\V(R)$
is irreducible if and only if\/ $\sing \V(R) = \kk\tt R$.
\end{proposition}
\begin{proof}
Recall that, by \reref{rtm2}, we have $\V(R) = \Ind_{\N_\K}^{\ti\K} R$.
The Lie algebra $\K$ is graded by the eigenspace decomposition of $\ad
\E'$. If $\k_n$ denotes the graded summand of eigenvalue $n$, then
one has the direct sum decomposition of Lie algebras
\begin{equation*}
\N_\K = \ti \dd \oplus \K_0' = \ti \dd \oplus \prod_{j\geq 0} \k_j
\end{equation*}
and the decomposition of vector spaces
\begin{equation*}
\ti \K
= (\k_{-2} \oplus \k_{-1}) \oplus \N_\K.
\end{equation*}

Since $\k_{-2} \oplus \k_{-1}$ is a graded Lie algebra,
its universal enveloping algebra is also graded.
Then $\V(R) = \Ind_{\N_\K}^{\ti\K} R$ is isomorphic to $\ue(\k_{-2} \oplus
\k_{-1})\tt R$, which can be endowed with a $\ZZ$-grading by
setting elements from $R$ to have degree zero, and elements from
$\k_{-i}$ to have degree $-i$. Thus submodules of $\V(R)$ contain
all $\E'$-homogeneous components of their elements, i.e., they are
graded submodules.

It is now easy to show that every homogeneous singular vector $v$,
say of degree $d<0$, is contained in some nonzero proper
$\ti\K$-submodule of $\V(R)$. Indeed $\ue(\ti \K) v =
\ue(\k_{-2} \oplus \k_{-1}) v$ is a nonzero submodule of $\V(R)$
lying in degrees $\leq d$, and it intersects $R$ trivially, since $R$
lies in degree zero.
\end{proof}

\subsection{Filtration of tensor modules}

After filtering the Lie algebra $\K$ using the contact filtration
$\fil'$ of $X$, it is convenient to filter tensor
$\Kd$-modules using the contact filtration of $H$. We therefore define
\begin{equation}
\fil'^p \V(R) = \fil'^p H \tt R \,, \qquad p = -1, 0,\dots \,.
\end{equation}
As usual, $\fil'^{-1} \V(R) = \{0\}$ and $\fil'^0 \V(R) = \kk \tt R$.
It will also be convenient to agree that $\fil'^{-2} \V(R) = \{0\}$.
The associated graded space is defined accordingly, and we have
isomorphisms of vector spaces
\begin{equation}
\gr'^p \V(R) \simeq \gr'^p H \tt R.
\end{equation}
Note that, since $\dd = \bar \dd \oplus \kk \d_0$ and the degree of
$\d_0$ equals two, $\gr'^p H$ is isomorphic to the direct sum
$\bigoplus_{i=0}^{\lfloor p/2 \rfloor} S^{p-2i} \bar \dd$. Here
$\lfloor p/2 \rfloor$ denotes the largest integer not greater than $p/2$,
which is $p/2$ for $p$ even and $(p-1)/2$ for $p$ odd.


\begin{lemma}\lbb{lfilact}
For every $p \geq 0$, we have{\rm:}
\begin{align}
\tag{i}
\bar\dd \cdot \fil'^p \V(R) &\subset \fil'^{p+1} \V(R) \,,
\\ \tag{ii}
\d_0 \cdot \fil'^p \V(R) &\subset \fil'^{p+2} \V(R) \,,
\\ \tag{iii}
\N_\K \cdot \fil'^p \V(R) &\subset \fil'^p \V(R) \,,
\\ \tag{iv}
\wti\K \cdot \fil'^p \V(R) &\subset \fil'^{p+2} \V(R) \,,
\\ \tag{v}
\K'_1 \cdot \fil'^p \V(R) &\subset \fil'^{p-1} \V(R) \,.
\end{align}
\end{lemma}
\begin{proof}
The proof of (i) and (ii) is clear, as the action of
elements in $\dd$ is by left multiplication on the left factor of
$\V(R) = H \tt R$. In particular, this implies $\dd\cdot \fil'^p
\V(R) \subset \fil'^{p+2} \V(R)$. Before proceeding with proving
(iii)--(v), observe that \eqref{tilded0},
\eqref{tildedi} imply
\begin{equation*}
\d^i \in \wti\dd + \K'_{-1}, \qquad \d_0 \in \wti\dd +
\K'_{-2},
\end{equation*}
so that $\bar\dd \subset \ti \dd + \K'_{-1}$. Also notice that $\ti
\K = \ti \dd + \K'_{-2}$, which implies $[\wti \K, \K'_p] \subset
\K'_{p-2}$. Moreover $\K'_{-1} + \N_\K = \bar \dd + \N_\K$, as
$\ti\dd \subset \N_\K$. Then we have:
\begin{align*}
[\bar \dd, \K'_1] & \subset [\wti \dd + \K'_{-1}, \K'_1] \subset
[\K'_{-1}, \K'_1] \subset
\K'_0 \subset \N_\K \,, \\
[\d_0, \K'_1] & \subset [\wti \K, \K'_1] \subset
\K'_{-1} \subset \bar \dd + \N_\K \,, \\
[\bar \dd, \N_\K] & \subset [\wti \dd + \K'_{-1}, \wti \dd +
\K'_0] \subset
\wti \dd + \K'_{-1} \subset \K'_{-1} + \N_\K \subset \bar \dd + \N_\K \,,\\
[\d_0, \N_\K] & \subset \wti \K = \dd + \N_\K \,.
\end{align*}

Now (iii) can be proved by induction as in the case of
$\Wd$ (see \cite[Lemma 6.3]{BDK1}), the basis of induction $p=0$
following from $\fil'^0 \V(R) \subset \sing \V(R)$. As for $p>0$,
notice that
$$\fil'^p \V(R)= \fil'^0 \V(R) + \bar \dd \fil'^{p-1} \V(R) + \d_0
\fil'^{p-2} \V(R).$$ Then:
\begin{align*}
\N_\K (\bar \dd \fil'^{p-1} \V(R)) & \subset \bar \dd (\N_\K
\fil'^{p-1} \V(R)) + [\bar \dd, \N_\K] \fil'^{p-1} \V(R)\\
& \subset \bar \dd \fil'^{p-1} \V(R) + (\bar \dd + \N_\K) \fil'^{p-1} \V(R)\\
& \subset \bar \dd \fil'^{p-1} \V(R) + \fil'^{p-1} \V(R) \subset
\fil'^p \V(R),
\end{align*}
and similarly
\begin{align*}
\N_\K (\d_0 \fil'^{p-2} \V(R)) & \subset \d_0 (\N_\K \fil'^{p-2}
\V(R)) + [\d_0, \N_\K] \fil'^{p-2} \V(R)\\
& \subset \dd \fil'^{p-2} \V(R) + (\dd + \N_\K) \fil'^{p-2} \V(R)\\
& \subset \dd \fil'^{p-2} \V(R) + \N_\K \fil'^{p-2} \V(R) \subset
\fil'^p \V(R).
\end{align*}
It is now immediate to prove (iv) from $\wti \K = \dd +
\N_\K$.

Finally, (v) can analogously be showed by induction on $p$:
when $p=0$, we have $\fil'^0 \V(R) \subset \sing \V(R)$, hence
$\K'_1 \fil'^0 \V(R) = \{0\}$ by definition of $\sing \V(R)$. When
$p>0$, we observe that
\begin{align*}
\K'_1 (\bar \dd \fil'^{p-1} \V(R)) & \subset \bar \dd (\K'_1
\fil'^{p-1}
\V(R)) + [\bar \dd, \K'_1] \fil'^{p-1} \V(R)\\
& \subset \bar \dd (\fil'^{p-2} \V(R)) + \N_\K \fil'^{p-1} \V(R)\\
& \subset \fil'^{p-1} \V(R),
\end{align*}
and that
\begin{align*}
\K'_1 (\d_0 \fil'^{p-2} \V(R)) & \subset \d_0 (\K'_1 \fil'^{p-2}
\V(R)) + [\d_0, \K'_1] \fil'^{p-2} \V(R)\\
& \subset \d_0 (\fil'^{p-3} \V(R)) + (\bar \dd + \N_\K) \fil'^{p-2} \V(R)\\
& \subset \fil'^{p-1} \V(R).
\end{align*}
This completes the proof.
\end{proof}

The above lemma implies that both $\N_\K$ and its quotient
$\N_\K/\K'_1 = \ti\dd \oplus \cspd$ act on each space $\gr'^p
\V(R)$. The next result describes the action of $\N_\K/\K'_1$
more explicitly.

\begin{lemma}\lbb{legraction}
The action of\/ $\ti \dd \simeq \dd$ and $\K'_0/\K'_1 \simeq \cspd =
\spd \oplus \kk I'$ on the space $\gr'^p \V(R) \simeq \gr'^p H \tt R$ is
given by{\rm:}
\begin{align}
\ti \d \cdot (f \tt u) & = f \tt \rho_R(\d) u,\\
A \cdot (\bar f\d_0^i \tt u) & = (A \bar f)\d_0^i \tt u + \bar
f\d_0^i \tt
\rho_R(A) u,\\
I' \cdot (f \tt u) & = p\, f \tt u + f \tt \rho_R(I') u,
\end{align}
where $A \in \spd, f\in \gr'^p H, \bar f \in S^{p-2i} \bar \dd, u
\in R$, and $A \bar f$ denotes the standard action of\/ $\spd \subset
\gld$ on\/ $\bar \dd$.
\end{lemma}
\begin{proof}
The proof is similar to that of Lemmas 6.4 and 6.5 from \cite{BDK1}.
\end{proof}

\begin{corollary}\lbb{grdecomposition}
We have an isomorphism of\/ $(\dd\oplus\cspd)$-modules
$$\gr'^p \V(\Pi, U, c) \simeq \bigoplus_{i=0}^{\lfloor p/2 \rfloor} \Pi \boxtimes
(S^{p-2i} \bar\dd \tt U, c + p).$$
\end{corollary}
\begin{proof}
Follows immediately from \leref{legraction}.
\end{proof}

\section{Tensor Modules of de Rham Type}\lbb{sdrm}

The main goal of this section is to define an important complex of
$\Kd$-modules, called the contact pseudo de Rham complex.
We continue to use the notation of Sections \ref{sfcc} and~\ref{sliedd}.

\subsection{The Rumin complex}\lbb{srumin}

As before, let $\th\in\dd^*$ be a contact form, and let
$\db\subset\dd$ be the kernel of $\th$.
Consider the wedge powers $\Om^n = \textstyle\bigwedge^n \dd^*$
and $\bar\Om^n = \textstyle\bigwedge^n \db^*$.
Then we have a short exact sequence
\begin{equation}\lbb{omom1}
0 \to \Th\Om^{n-1} \to \Om^n \to \bar\Om^n \to 0 \, ,
\end{equation}
where $\Th$ is the operator of left wedge multiplication with
$\th$, i.e., $\Th(\al) = \th\wedge\al$. For $\al\in\Om^n$, we will
denote by $\bar\al\in\bar\Om^n$ its projection via \eqref{omom1}.

The direct sum decomposition $\dd=\db\oplus\kk s$ gives
a splitting of the sequence \eqref{omom1}.
In more detail, elements $\bar\al\in\bar\Om^n$ are
identified with $n$-forms $\al\in\Om^n$ such that $\io_s\al=0$.
Thus we have a direct sum $\Om^n = \Th\Om^{n-1} \oplus \bar\Om^n$.
Then $\Th^2=0$ implies that
$\ker\Th|_{\Om^n} = \Th\Om^{n-1}$, and we get a natural
isomorphism
\begin{equation}\lbb{omom2}
\Th\Om^n \isoto \bar\Om^n , \quad \th\wedge\al \mapsto \bar\al \,.
\end{equation}

The $2$-form $\om=\diz\th$ can be identified with $\bar\om$,
because $\io_s\om=0$.
Denote by $\Psi$ (respectively, $\bar\Psi$)
the operator of left wedge multiplication with $\om$
(respectively, $\bar\om$). Consider the images and kernels of $\bar\Psi$:
\begin{equation}\lbb{baribark}
\bar I^n = \bar\Psi\bar\Om^{n-2} \subset\bar\Om^n \,, \qquad
\bar K^n = \ker \bar\Psi|_{\bar\Om^n} \subset\bar\Om^n \,.
\end{equation}
Since $\bar\om$ is nondegenerate, we have $\bar I^n=\bar\Om^n$ for $n\ge N+1$ and
$\bar K^n=0$ for $n\le N-1$. In particular,
$\bar\Psi\colon\bar\Om^{N-1}\to\bar\Om^{N+1}$ is an isomorphism. More
generally, for all $m=0,\dots,N$, the maps
$\bar\Psi^m\colon\bar\Om^{N-m}\to\bar\Om^{N+m}$ are isomorphisms.

\begin{lemma}\lbb{ljnomi}
The composition of natural maps $\bar K^{N} \injto \bar\Om^{N} \surjto
\bar\Om^{N}/\bar I^{N}$ is an isomorphism. More generally, the composition
\begin{equation*}
\bar K^{N+m} \injto \bar\Om^{N+m} \xrightarrow{(\bar\Psi^m)^{-1}}
\bar\Om^{N-m} \surjto \bar\Om^{N-m}/\bar I^{N-m}
\end{equation*}
is an isomorphism for all\/ $m=0,\dots,N$.
\end{lemma}
\begin{proof}
To show surjectivity, take any $\al\in\bar\Om^{N-m}$. We want to find
$\be\in \bar K^{N+m}$ such that $\al - (\bar\Psi^m)^{-1}\be \in
\bar I^{N-m}$. Since $\bar\Psi^{m+2}\colon\bar\Om^{N-m-2}\to\bar\Om^{N+m+2}$ is an isomorphism, there is
$\ga\in\bar\Om^{N-m-2}$ such that $\bar\Psi^{m+2}\ga = \bar\Psi^{m+1}\al$.
Then $\be = \bar\Psi^{m}(\al-\bar\Psi\ga)$ satisfies the above conditions.

To prove injectivity, we need to show that $\bar\Psi^{m}\bar I^{N-m}\cap
\bar K^{N+m} = \{0\}$.
If $\al\in \bar\Psi^{m}\bar I^{N-m}\cap \bar K^{N+m}$, then
$\al=\bar\Psi^{m+1}\rho$ for some $\rho\in\bar\Om^{N-m-2}$. But then
$\bar\Psi^{m+2}\rho = \bar\Psi\al = 0$, which implies $\rho=0$ and
$\al=0$.
\end{proof}

Since $A\cdot\bar\om=0$ for $A\in\spd$ and the action of $A$ is
an even derivation of the wedge product (see \leref{lglom} and \eqref{acder}),
it follows that $\bar I^n$ and $\bar K^n$ are $\spd$-submodules of
$\bar\Om^n$. Furthermore, the map $\bar\Psi$ is an $\spd$-homomorphism.
In particular, the isomorphism from \leref{ljnomi} commutes with the
action of $\spd$. Recall that $R(\pi_n)$ denotes the $n$-th fundamental
representation of $\spd$, and $R(\pi_0) = \kk$.

\begin{lemma}\lbb{ljnomi2}
We have isomorphisms of\/ $\spd$-modules
\begin{equation*}
\bar\Om^{n}/\bar I^{n} \simeq \bar K^{2N-n} \simeq R(\pi_n) \,, \qquad
0\leq n \leq N \,.
\end{equation*}
\end{lemma}
\begin{proof}
This is well known; see, e.g., \cite[Lecture 17]{FH}.
\end{proof}

Following \cite{Ru}, we consider the spaces
\begin{equation}\lbb{inkn1}
I^n = \Psi\Om^{n-2} + \Th\Om^{n-1} \subset \Om^n , \qquad K^n =
\ker\Psi|_{\Om^n} \cap \ker\Th|_{\Om^n} \subset \Om^n.
\end{equation}
Using $\Th\Psi = \Psi\Th$ and \eqref{omom1}, we obtain a short
exact sequence
\begin{equation}\lbb{iibar}
0 \to \Th\Om^{n-1} \to I^n \to \bar I^n \to 0 \, ,
\end{equation}
while \eqref{omom2} gives a natural isomorphism
\begin{equation}\lbb{kkbar}
K^n \isoto \bar K^{n-1} , \quad \th\wedge\al \mapsto \bar\al \, .
\end{equation}
The above equations imply that
$I^n = \Om^n$ for $n\ge N+1$ and $K^n = 0$ for $n\le N$. It
is also clear that $\Om^n / I^n \simeq \bar\Om^n / \bar I^n$ for
all $n$.

The ``constant-coefficient''
\emph{Rumin complex} \cite{Ru} is the following complex of
$\csp\,\db$-modules
\begin{equation}\lbb{rumk1}
0 \to \Om^0 / I^0 \xrightarrow{\diz} \cdots
\xrightarrow{\diz} \Om^{N} / I^{N} \xrightarrow{\dizru}
K^{N+1} \xrightarrow{\diz} \cdots \xrightarrow{\diz}
K^{2N+1} \,,
\end{equation}
where the map $\dizru$ is defined as in \cite{Ru}.
We will need the ``pseudo'' version of this complex
defined in \seref{sdrmk} below. The latter is a contact counterpart of the
pseudo de Rham complex from \cite{BDK,BDK1}, which we review
in the next subsection.

\subsection{Pseudo de Rham complex}\lbb{sdrmw}
Following \cite{BDK}, we define the spaces of \emph{pseudoforms\/}
$\Om^n(\dd) = H\tt\Om^n$ and $\Om(\dd) = H\tt\Om =
\textstyle\bigoplus_{n=0}^{2N+1} \Om^n(\dd)$.
They are considered as $H$-modules, where $H$ acts on the first
factor by left multiplication.
We can identify $\Om^n(\dd)$ with the space of linear maps from
$\bigwedge^n \dd$ to $H$, and $H^{\tt2}\tt_H \Om^n(\dd)$ with
$\Hom(\bigwedge^n \dd, H^{\tt2})$. For $g\in H$, $\al\in\Om$, we
will write the element $g\tt\al\in\Om(\dd)$ as $g\al$; in
particular, we will identify $\Om$ with
$\kk\tt\Om\subset\Om(\dd)$.

Let us consider $H=\ue(\dd)$ as a left $\dd$-module with respect
to the action $a\cdot h = -ha$, where $ha$ is the product of
$a\in\dd\subset H$ and $h\in H$ in $H$. Then consider the cohomology
complex of $\dd$ with coefficients in $H$:
\begin{equation}\lbb{domd}
0 \to \Om^0(\dd) \xrightarrow{\di} \Om^1(\dd) \xrightarrow{\di}
\cdots \xrightarrow{\di} \Om^{2N+1}(\dd) \,.
\end{equation}
Explicitly, the \emph{differential} $\di$ is given by the formula
($\al\in\Om^n(\dd)$, $a_i\in\dd$):
\begin{equation}\lbb{dw}
\begin{split}
&\begin{split} (\di \al)(a_1 & \wedge \dots \wedge a_{n+1})
\\
&= \sum_{i<j} (-1)^{i+j} \al([a_i,a_j] \wedge a_1 \wedge \dots
\wedge \what a_i \wedge \dots \wedge \what a_j \wedge \dots \wedge
a_{n+1})
\\
&+ \sum_i (-1)^i \al(a_1 \wedge \dots \wedge \what a_i \wedge
\dots \wedge a_{n+1}) \, a_i \qquad\text{if}\;\; n\ge1,
\end{split}
\\
&(\di \al)(a_1) = -\al a_1 \qquad\text{if}\;\; \al\in\Om^0(\dd)=H.
\end{split}
\end{equation}
Notice that $\di$ is $H$-linear.
The sequence \eqref{domd} is called the \emph{pseudo de Rham complex}.
It was shown in \cite[Remark 8.1]{BDK} that the
$n$-th cohomology of the complex $(\Om(\dd), \di)$ is trivial
for $n\ne 2N+1=\dim\dd$, and it is $1$-dimensional for $n=2N+1$. In
particular, the sequence \eqref{domd} is exact.

\begin{example}\lbb{ed1}
For $\al=1 \in H=\Om^0(\dd)$, Eq.\ \eqref{dw} gives
\begin{equation}\lbb{di1}
-\di 1 = \eps := \sum_{i=0}^{2N} \d_i \tt x^i \in H\tt\dd^* = \Om^1(\dd) \,.
\end{equation}
\end{example}


Next, we introduce $H$-bilinear maps
\begin{equation}\lbb{*n}
* \colon \Wd \tt \Om^n(\dd) \to H^{\tt2} \tt_H \Om^n(\dd)
\end{equation}
by the formula \cite{BDK}:
\begin{equation}\lbb{fa*om}
\begin{split}
(w*\ga)(a_1 \wedge & \dots \wedge a_n) = -(f \tt g a) \, \al(a_1
\wedge \dots \wedge a_n)
\\
&+ \sum_{i=1}^n (-1)^i (f a_i \tt g) \, \al(a \wedge a_1 \wedge
\dots \wedge \what a_i \wedge \dots \wedge a_n)
\\
&+ \sum_{i=1}^n (-1)^i (f \tt g) \, \al([a,a_i] \wedge a_1 \wedge
\dots \wedge \what a_i \wedge \dots \wedge a_n) \in H^{\tt2} \,,
\end{split}
\end{equation}
where $n\ge1$, $w=f\tt a\in\Wd$ and $\ga=g\al\in\Om^n(\dd)$.
When $\ga=g\in\Om^0(\dd)=H$, we let $w*\ga = -f\tt ga$.
Note that the latter coincides with the action \eqref{wdac*} of
$\Wd$ on $H$.

It was shown in \cite{BDK,BDK1} that
maps \eqref{*n} provide each $\Om^n(\dd)$ with a structure of
a $\Wd$-module. These modules are instances of \emph{tensor modules}
as introduced in \cite{BDK1},
namely $\Om^n(\dd)=\T(\kk,\Om^n)$ (see \seref{stwd}).
The action of\/ $\Wd$ commutes with $\di$, i.e.,
\begin{equation}\lbb{di*}
w*(\di\ga) = ((\id\tt\id)\tt_H \di)(w*\ga)
\end{equation}
for $w\in\Wd$, $\ga\in\Om^n(\dd)$.


Let us extend the wedge product in $\Om$ to a product in
$\Om(\dd)$ by setting
\begin{equation*}
(f\al)\wedge(g\be) = (fg)(\al\wedge\be) \,, \qquad
\al,\be\in\Om \,, \;\; f,g\in H \,.
\end{equation*}
In a similar way, we also extend it to products
\begin{align*}
(h\tt_H(f\tt\al))\wedge\be &= h\tt_H(f\tt(\al\wedge\be)) \,,
\\
\al\wedge(h\tt_H(g\tt\be)) &= h\tt_H(g\tt(\al\wedge\be)) \,,
\qquad h\in H^{\tt2} \,.
\end{align*}


\begin{lemma}\lbb{l*diio}
For any $w\in\Wd$, $\al\in\Om^n$ and $\be\in\Om$, we
have{\rm:}
\begin{align}
\lbb{dider}
\di(\al\wedge\be) &= \diz\al\wedge\be + (-1)^n
\al\wedge\di\be \,,
\\
\lbb{*der}
w*(\al\wedge\be) &=
(w*\al)\wedge\be + \al\wedge(w*\be) + w \tt_H (\al\wedge\be)\,.
\end{align}
\end{lemma}
\begin{proof}
Since $\diz$ is an odd derivation of the wedge product,
by subtracting \eqref{dizder} from \eqref{dider}, we obtain that
\eqref{dider} is equivalent to:
\begin{equation*}
(\di-\diz)(\al\wedge\be) = (-1)^n \al\wedge (\di-\diz)\be \,.
\end{equation*}
On the other hand,
comparing \eqref{dider} with \eqref{dizder} and \eqref{wedge}, we see that
\begin{equation}\lbb{di1al}
(\di-\diz)\al = -\eps\wedge \al \,,
\end{equation}
where $\eps$ is defined by \eqref{di1}. Then \eqref{dider} follows from
the associativity and graded-commutativity of the wedge product
(see \eqref{omasgc}).

By $H$-linearity, it is enough to prove \eqref{*der} in the case
$w=1\tt\d_i$. Then by \eqref{wdgcd3} we have
\begin{equation*}
(1\tt\d_i)*\al - (1\tt\d_i) \tt_H \al
= (1\tt1) \tt_H (\ad\d_i)\cdot\al
+ \sum_{j=0}^{2N} (\d_j\tt1) \tt_H e_i^j \cdot \al \,.
\end{equation*}
Using that $(\ad\d_i)$ and $e_i^j$ are even derivations of the wedge product
(see \eqref{acder}) completes the proof.
\end{proof}

\subsection{Contact pseudo de Rham complex}\lbb{sdrmk}
As before, let $\Om^n(\dd)=H\tt\Om^n$, $\Om(\dd) =
\bigoplus_{n=0}^{2N+1} \Om^n(\dd)$ be the spaces of pseudoforms.
We extend the operators $\Th$ and $\Psi$ defined in \seref{srumin}
to $\Om(\dd)$ by
$H$-linearity. We also set
$I^n(\dd) = H\tt I^n$ and $K^n(\dd) = H\tt K^n$.
{}From \eqref{dider} and $\om=\diz\th$, we deduce:
\begin{equation}\lbb{dipsth}
\di\Psi=\Psi\di, \quad \di\Th=\Psi-\Th\di \,,
\end{equation}
where $\di$ is given by \eqref{dw}.
This implies that $\di I^n(\dd) \subset I^{n+1}(\dd)$ and $\di
K^n(\dd) \subset K^{n+1}(\dd)$. Therefore, we have the induced
complexes
\begin{align}
\lbb{domd3} 0 \to \Om^0(\dd) / I^0(\dd) &\xrightarrow{\di}
\Om^1(\dd) / I^1(\dd) \xrightarrow{\di} \cdots \xrightarrow{\di}
\Om^{N}(\dd) / I^{N}(\dd) \intertext{and} \lbb{domd4} K^{N+1}(\dd)
&\xrightarrow{\di} K^{N+2}(\dd) \xrightarrow{\di} \cdots
\xrightarrow{\di} K^{2N+1}(\dd).
\end{align}

\begin{lemma}[cf.\ \cite{Ru}]\lbb{ldomd3}
The sequences \eqref{domd3} and \eqref{domd4} are exact.
\end{lemma}
\begin{proof}
First, to show exactness at the term $\Om^n(\dd) / I^n(\dd)$ in
\eqref{domd3} for $n\le N-1$, take $\al\in\Om^n(\dd)$ such that
$\di\al\in I^{n+1}(\dd)$. This means $\di\al = \Th\be + \Psi\ga$
for some $\be\in\Om^{n}(\dd)$, $\ga\in\Om^{n-1}(\dd)$. Then
$\di(\al-\Th\ga) = \di\al - \Psi\ga + \Th\di\ga =
\Th(\be+\di\ga)$\,; hence, by changing the representative $\al$
mod $I^n(\dd)$, we can assume that $\ga=0$. Now we have $0 =
\di^2\al = \di\Th\be = \Psi\be - \Th\di\be$. Then $\Psi(\Th\be) =
0$, i.e., $\Th\be \in K^{n+1}(\dd)$. But $K^{n+1}(\dd) = 0$ for
$n\le N-1$; thus $\Th\be=0$ and $\di\al = 0$. It follows that
$\al=\di\rho$ for some $\rho\in\Om^{n-1}(\dd)$.

To prove exactness at the term $K^n(\dd)$ in \eqref{domd4} for
$n\ge N+2$, take $\al\in K^n(\dd)$ such that $\di\al=0$. Then
$\al=\di\be$ for some $\be\in\Om^{n-1}(\dd)$. Since $I^{n-1}(\dd)
= \Om^{n-1}(\dd)$ for $n\ge N+2$, we can write $\be = \Th\ga +
\Psi\rho$ for some $\ga\in\Om^{n-2}(\dd)$,
$\rho\in\Om^{n-3}(\dd)$. But since $\di(\Psi\rho) =
\di(\Th\di\rho)$, by replacing $\ga$ with $\ga+\di\rho$, we can
assume that $\rho=0$. Then $\di\be = -\Th\di\ga + \Psi\ga$ and
$\Th\al=0$ implies $\Th\Psi\ga = 0$. Therefore, $\be=\Th\ga \in
K^{n-1}(\dd)$, which completes the proof.
\end{proof}

Now, following \cite{Ru}, we will construct a map $\diru\colon
\Om^{N}(\dd) / I^{N}(\dd) \to K^{N+1}(\dd)$ that connects the
complexes \eqref{domd3} and \eqref{domd4}, which we will
call the \emph{Rumin map}.
Since $I^{N+1}(\dd) = \Om^{N+1}(\dd)$,
for every $\al\in\Om^{N}(\dd)$ we can write $\di\al =
\Th\be + \Psi\ga$ for some $\be\in\Om^{N}(\dd)$,
$\ga\in\Om^{N-1}(\dd)$. Then, as in the proof of \leref{ldomd3},
we have $\di\ti\al = \Th(\be+\di\ga) \in K^{N+1}(\dd)$ for
$\ti\al=\al-\Th\ga$.
We let $\diru\al=\di\ti\al$.

We have to prove that $\diru\al$ is independent of the choice of
$\ti\al$ and depends only on the class of $\al$ mod $I^{N}(\dd)$.
First, if $\di\al = \Th\be + \Psi\ga = \Th\be' + \Psi\ga'$, then
$\Th\Psi(\ga-\ga') = 0$, which implies $\Th(\ga-\ga') \in
K^{N}(\dd)$. But $K^{N}(\dd) = 0$\,; hence, $\ti\al = \al-\Th\ga =
\al-\Th\ga'$. Next, consider the case when $\al\in I^{N}(\dd)$.
Write $\al = \Th\mu + \Psi\rho$\,; then $\di\al = \Th(-\di\mu) +
\Psi(\mu+\di\rho)$ and $\diru\al = \Th( (-\di\mu) +
\di(\mu+\di\rho) ) = 0$, as desired.

Using the Rumin map $\diru$,
we can combine the two complexes \eqref{domd3} and \eqref{domd4}.

\begin{proposition}[cf.\ \cite{Ru}]\lbb{pdomd4}
The sequence
\begin{equation*}
0 \to \Om^0(\dd) / I^0(\dd) \xrightarrow{\di} \cdots
\xrightarrow{\di} \Om^{N}(\dd) / I^{N}(\dd) \xrightarrow{\diru}
K^{N+1}(\dd) \xrightarrow{\di} \cdots \xrightarrow{\di}
K^{2N+1}(\dd)
\end{equation*}
is an exact complex.
\end{proposition}
\begin{proof}
In the preceding discussion we have shown that $\diru$ is well
defined. Next, it is clear by construction that $\diru\di=0$ and
$\di\diru=0$. Due to \leref{ldomd3}, it remains only to check
exactness at the terms $\Om^{N}(\dd) / I^{N}(\dd)$ and
$K^{N+1}(\dd)$.

First, let $\al\in\Om^{N}(\dd)$ be such that $\diru\al = 0$. Then
$\di\ti\al = \diru\al = 0$\,; hence $\ti\al=\di\be$ for some
$\be\in\Om^{N-1}(\dd)$. Then $\al+I^{N}(\dd) = \ti\al+I^{N}(\dd) =
\di(\be+I^{N-1}(\dd))$.

Now let $\al\in K^{N+1}(\dd)$ be such that $\di\al = 0$. Then
$\al=\di\be$ for some $\be\in\Om^{N}(\dd)$. Since $\di\be\in
K^{N+1}(\dd)$, we can take $\ti\be=\be$, and
$\diru\be=\di\ti\be=\al$.
\end{proof}

We will call the complex from \prref{pdomd4}
the \emph{contact pseudo de Rham complex}.

\subsection{$\Kd$-action on the contact pseudo de Rham complex}\lbb{spsrumin}
Here we prove that the contact pseudo de Rham complex is a complex of
$\Kd$-modules, and we realize its members as tensor modules.

First, we show that the
members of the Rumin complex \eqref{rumk1} are $\cspd$-modules.
Recall that the Lie algebra $\gld$ acts on the space $\Om^n$
of constant coefficient $n$-forms via \eqref{acdotal}, and this action
is by even derivations (see \eqref{acder}).

\begin{lemma}\lbb{ldomd5}
For every $n$, we have{\rm:} $\cspd\cdot I^n \subset I^n$
and $\cspd\cdot K^n \subset K^n$. In addition, $\cz\cdot \Om^n \subset I^n$
and $\cz\cdot K^n = \{0\}$.
Hence the\/ $\gld$-action on $\Om^n$ induces actions of\/ $\cspd$ on
$\Om^n/I^n$ and $K^n$, and the trivial action of\/ $\cz$ on them.
\end{lemma}
\begin{proof}
By \leref{lglom}, $A\cdot\al=c\al$ for $A\in\cspd$, $\al\in\{\th,\om\}$
and some $c\in\CC$. Then by \eqref{acder},
$A\cdot(\al\wedge\be) = \al\wedge(c\be+A\cdot\be)$ for all $\be\in\Om$.
This implies $A\cdot I^n \subset I^n$ and $A\cdot K^n \subset K^n$.

Next, recall that $\cz= \Span\{ e_k^0 \}_{k\ne0}$ and
$e_k^0 \cdot x^i = -\de_k^i x^0 = \de_k^i \th$. Then
\begin{equation*}
e_k^0 \cdot (x^{i_1} \wedge\cdots\wedge x^{i_n})
= \th \wedge x^{i_2} \wedge\cdots\wedge x^{i_n} \,,
\qquad\text{if}\quad k=i_1 \,,
\end{equation*}
and is zero if $k\ne i_s$ for all $s$. Therefore,
$\cz\cdot \Om^n \subset \Th\Om^{n-1} \subset I^n$.

Now, if $\al\in K^n$, by \eqref{kkbar} we can write $\al=\th\wedge\be$
for some $\be\in\Om^{n-1}$. Then for $k\ne0$ we have
$e_k^0 \cdot\be = \th\wedge\ga$ for some $\ga\in\Om^{n-2}$, and we find
\begin{equation*}
e_k^0 \cdot\al = e_k^0 \cdot(\th\wedge\be) = \th \wedge (e_k^0 \cdot\be)
= \th \wedge (\th\wedge\ga) = 0 \,,
\end{equation*}
using that $e_k^0 \cdot\th = 0$.
\end{proof}
\begin{lemma}\lbb{ldomd6}
We have isomorphisms of\/ $\cspd$-modules
\begin{equation*}
\Om^{n}/I^{n} \simeq (R(\pi_n),-n) \,, \quad
K^{2N+1-n} \simeq (R(\pi_n), -2N-2+n) \,, \qquad
0\leq n \leq N \,.
\end{equation*}
\end{lemma}
\begin{proof}
Recall that we have isomorphisms of $\spd$-modules
$\Om^n / I^n \simeq \bar\Om^n / \bar I^n$ and
$K^n \simeq \bar K^{n-1}$ (see \eqref{kkbar}).
The $\spd$-action on these modules is described in \leref{ljnomi2}.
Finally, to determine the action of $I'$, we use \eqref{acder},
\eqref{glom3} and \eqref{kkbar}. We obtain that $I'$ acts as $-n$
on $\bar\Om^n \subset \Om^n$ and as $-n-1$ on $K^n$.
\end{proof}

Here is the main result of this section.

\begin{theorem}\lbb{tdomd7}
The contact pseudo de Rham complex
\begin{equation*}
0 \to \Om^0(\dd) / I^0(\dd) \xrightarrow{\di} \cdots
\xrightarrow{\di} \Om^{N}(\dd) / I^{N}(\dd) \xrightarrow{\diru}
K^{N+1}(\dd) \xrightarrow{\di} \cdots \xrightarrow{\di}
K^{2N+1}(\dd)
\end{equation*}
is an exact complex of\/ $\Kd$-modules. Its members are tensor
modules, namely $$\Om^n(\dd) / I^n(\dd) = \T(\kk,\Om^n/I^n) =
\T(\kk, R(\pi_n), -n)$$ and $$K^n(\dd) = \T(\kk,K^n) = \T(\kk,
R(\pi_{2N+1-n}), -n-1).$$
\end{theorem}
\begin{proof}
Recall that all $\Om^n(\dd)=\T(\kk,\Om^n)$ are tensor modules for $\Wd$;
see \seref{stwd} and \cite{BDK1}. In particular, $e*(1\tt\al)$
is given by \reref{rtm} for $\al\in\Om^n$. By \leref{ldomd5},
$\cz$ acts trivially on $K^n$ and $K^n$ is a $\cspd$-module.
Therefore, for $\al\in K^n$, the action $e*(1\tt\al)$ is given by \eqref{ktm}.
By definition, this means that $K^n(\dd) = H\tt K^n = \T(\kk,K^n)$
has the structure of a tensor $\Kd$-module.
The same argument applies to the quotient
$\Om^n(\dd) / I^n(\dd) = \T(\kk,\Om^n/I^n)$.

The exactness of the complex was established in \prref{pdomd4}. It remains
to prove that the maps of the complex are homomorphisms of $\Kd$-modules.
For $\di$, this follows by construction from the fact that
$\di\colon\Om^n(\dd) \to \Om^{n+1}(\dd)$ is a homomorphism of $\Wd$-modules.
In order to prove it for $\diru$, we need the next lemma, which can be deduced
from \reref{rtm} and \leref{lglom}.

\begin{lemma}\lbb{ldomd8}
Identifying\/ $\al\in\Om^n$ with\/ $1\tt\al \in \Om^n(\dd) = H\tt\Om^n$,
we have{\rm:}
\begin{align}\lbb{eth}
e*\th &= -(e+\d_0\tt1) \tt_H \th \,,
\\ \lbb{eom}
e*\om &= -(e+\d_0\tt1) \tt_H \om
- \sum_{i=1}^{2N}\, (\d_i\d_0 \tt 1) \tt_H (\th\wedge x^i) \,.
\end{align}
\end{lemma}

Now take an $\al\in\Om^{N}(\dd)$ and write
$\di\al = \Th\be + \Psi\ga = \th\wedge\be + \om\wedge\ga$.
Then, by definition, $\diru\al = \di(\al-\th\wedge\ga)$.
Using that $\di$ is a homomorphism (see \eqref{di*}), we obtain
\begin{equation*}
e*(\diru\al) = \bigl((\id\tt\id)\tt_H \di\bigr)
\bigl( e*\al - e*(\th\wedge\ga) \bigr) \,.
\end{equation*}
Then we find from \eqref{*der} and \eqref{eth} that
\begin{equation*}
e*(\th\wedge\ga) = \th\wedge\ga' \,, \qquad
\ga' = e*\ga - (\d_0\tt1) \tt_H \ga \,.
\end{equation*}
On the other hand, using again \eqref{*der}, \eqref{eth} and \eqref{eom},
we compute
\begin{equation*}
((\id\tt\id)\tt_H \di)(e*\al) = e*(\di\al)
= e*(\th\wedge\be) + e*(\om\wedge\ga)
= \th\wedge\be' + \om\wedge\ga'
\end{equation*}
for some $\be'$, where $\ga'$ is as above.
Then
\begin{equation*}
((\id\tt\id)\tt_H \diru)(e*\al)
= \bigl((\id\tt\id)\tt_H \di\bigr)
\bigl( e*\al - \th\wedge\ga' \bigr) \,,
\end{equation*}
which coincides with $e*(\diru\al)$. This completes the proof of the theorem.
\end{proof}

\subsection{Twisted contact pseudo de Rham complex}\lbb{stwistedrumin}

For any choice of a finite-dimensional $\dd$-module $\Pi$, one may
apply the twisting functor $T_\Pi$ from \seref{stwrep}
to \thref{tdomd7} and obtain a
corresponding exact complex of $\Kd$-modules
\begin{align*}
0 &\to \T(\Pi, \kk, 0) \xrightarrow{\di_\Pi} \T(\Pi, R(\pi_1), -1)
\xrightarrow{\di_\Pi} \cdots \xrightarrow{\di_\Pi} \T(\Pi, R(\pi_N),
-N) \xrightarrow{\diru_\Pi}\\
& 
\T(\Pi, R(\pi_N), -N-2)
\xrightarrow{\di_\Pi} \cdots \xrightarrow{\di_\Pi} \T(\Pi, R(\pi_1),
-2N-1) \xrightarrow{\di_\Pi} \T(\Pi, \kk, -2N-2),
\end{align*}
where we used the notation $\di_\Pi = T_\Pi(\di)$ and $\diru_\Pi =
T_\Pi(\diru)$. In the rest of the paper, we will suppress the
reference to $\Pi$, and write $\di$ instead of $\di_\Pi$ and $\diru$
instead of $\diru_\Pi$ whenever there is no possibility of
confusion. If we set
\begin{equation}\lbb{notation1}
\begin{split}
\V^\Pi_p & = \V(\Pi, R(\pi_p), p) = \T(\Pi \otimes \kk_{\tr\ad}, R(\pi_p), p-2N-2)\\
\V^\Pi_{2N+2-p} & = \V(\Pi, R(\pi_p), 2N+2-p) =
\T(\Pi \otimes \kk_{\tr\ad}, R(\pi_p), -p),
\end{split}
\end{equation}
for $0 \leq p \leq N$, where $R(\pi_0) = \kk$ denotes the trivial
representation of $\spd$, then we obtain an exact sequence of
$\Kd$-modules
\begin{equation}\lbb{tensorcomplex}
\begin{split}
0 \to \V^\Pi_{2N+2} \xrightarrow{\di}& \V^\Pi_{2N+1}
\xrightarrow{\di} \cdots \xrightarrow{\di} \V^\Pi_{N+2}
\xrightarrow{\diru} \V^\Pi_N \xrightarrow{\di} \cdots
\xrightarrow{\di} \V^\Pi_1 \xrightarrow{\di} \V^\Pi_0.
\end{split}
\end{equation}
The above exact complex will be useful in the study of reducible
tensor modules and in the computation of their singular vectors. We
will be using notation \eqref{notation1}
throughout the rest of the paper. Notice that $\V^\Pi_{N+1}$ is not
defined.

\section{Irreducibility of Tensor Modules}\lbb{sirrkmod}


We will investigate submodules of tensor modules, and prove
a criterion for irreducibility of tensor modules.
Throughout the section, $R$ will be an irreducible $(\dd \oplus
\cspd)$-module with an action denoted $\rho_R$,
and $\V(R)$ the corresponding tensor module.

\subsection{Coefficients of elements and submodules}\lbb{skirten}
Note that every element $v\in \V(R) = H \tt R$ can be written
uniquely in the form
\begin{equation}\lbb{coeff1}
v= \sum_{I \in \ZZ_+^{2N+1}} \d^{(I)} \tt v_I \,, \qquad v_I \in R \,.
\end{equation}

\begin{definition}
The nonzero elements $v_I$ in \eqref{coeff1} are called {\em
coefficients} of $v\in \V(R)$. For a submodule $M \subset \V(R)$,
we denote by $\coef M$ the subspace of $R$ linearly spanned by
all coefficients of elements from $M$.
\end{definition}

It will be convenient to introduce the notation
\begin{equation}\lbb{psiu}
\psi(u) = \sum_{i,j =1}^{2N} \d_i \d_j \tt \rho_R(f^{ij}) u \,,
\qquad u \in R \,.
\end{equation}

\begin{lemma}\lbb{lcoefficients}
If\/ $v\in \V(R)$ is given by \eqref{coeff1}, then
\begin{equation}\lbb{coeff2}
\begin{split}
e*v = & \sum_I (1 \tt \d^{(I)}) \tt_H \psi(v_I)\\
&+ \text{\rm{ terms in }}
(\kk \tt \d^{(I)}H) \tt_H (\fil^1 H \tt (\kk +
\rho_R(\spd + \dd)) \cdot v_I) \,.
\end{split}
\end{equation}
In particular, the coefficient multiplying $1 \tt \d^{(I)}$ equals
$\psi(v_I)$ modulo $\fil^1 \V(R)$.
\end{lemma}
\begin{proof}
We rewrite \eqref{ksing1} using the fact that
\begin{equation}\lbb{didiv}
(\d_i\tt1) \tt_H v = (1\tt1)\tt_H \d_i v - (1\tt\d_i) \tt_H v
\end{equation}
for any $v\in\V(R)$. We obtain:
\begin{equation}\lbb{ksing2}
\begin{split}
e * (&1 \tt u) = ( 1 \tt 1) \tt_H \Bigl( \psi(u) - \sum_{k =1}^{2N}
\d_k
\tt \rho_R(\d^k)u \Bigr)\\
&+ \mbox{ terms in } (\kk \tt H) \tt_H \Bigl(\dd \tt \bigl(\kk +
\rho_R(\spd)\bigr) \cdot u + \kk \tt \bigl(\kk + \rho_R(\spd +
\dd)\bigr) \cdot u\Bigr) \,.
\end{split}
\end{equation}
Then plugging in \eqref{coeff1} and applying $H$-bilinearity completes the proof.
\end{proof}

\begin{remark}
For $v \in \sing \V(R)$, we have by \eqref{Kactionsing}
\begin{equation}\lbb{ksing3}
e * v = \sum_{i,j=1}^{2N} (1 \tt \d_i\d_j) \tt_H \rho_{\sing} (f^{ij}) v
+ \mbox{ terms in } (\kk \tt \fil^1 H) \tt_H \V(R).
\end{equation}
\end{remark}

\begin{lemma}
For any nonzero proper $\Kd$-submodule $M \subset \V(R)$, we have
$\coef M = R$.
\end{lemma}
\begin{proof}
Pick a nonzero element $v = \sum_I \d^{(I)} \tt v_I$ contained in
$M$. Then \leref{lcoefficients} shows that $M$ contains an element
congruent to $\psi(v_I)$ modulo $\fil^1 \V(R)$, thus coefficients of
$\psi(v_I)$ lie in $\coef M$ for all $I$. This proves that
$\spd(\coef M) \subset \coef M$. Similarly, one can write
\begin{equation*}
\begin{split}
e*v &= \sum_I \, (1 \tt \d^{(I)}) \tt_H \Bigl( \psi(v_I) - \sum_{k
=1}^{2N}
\d_k \tt \rho_R( \d^k)v_I \Bigr) \\
&+ \mbox{ terms in } (\kk \tt \d^{(I)} H) \tt_H \Bigl( \dd  \tt
\rho_R(\spd + \kk) v_I + \kk \tt \bigl(\kk + \rho_R(\spd +
\dd)\bigr) v_I \Bigr),
\end{split}
\end{equation*}
showing that $\rho_R(\d^k)v_I \in \coef M$ for all $I$ and all $k=1,\dots,
2N$. Thus, $\db (\coef M) \subset \coef M$. However, $\db$
generates $\dd$ as a Lie algebra, hence $\dd$ stabilizes $\coef M$
as well. Then $\coef M$ is a nonzero $(\dd \oplus
\cspd)$-submodule of $R$. Irreducibility of $R$ now gives
that $\coef M = R$.
\end{proof}

\begin{corollary}\lbb{cpsiu}
Let\/ $M$ be a nonzero proper $\Kd$-submodule of\/ $\V(R)$. Then
for every $u \in R$ there is an element in $M$ that coincides with
$\psi(u)$ modulo $\fil^1 \V(R)$.
\end{corollary}
\begin{proof}
As $\coef M = R$, it is enough to prove the statement for
coefficients of elements $v\in M$. Since $M$ is a
$\Kd$-submodule of $\V(R)$, the coefficient multiplying $1 \tt
\d^{(I)}$ in \eqref{coeff2} still lies in $M$ and it equals
$\psi(v_I)$ modulo $\fil^1 \V(R)$.
\end{proof}

\subsection{An irreducibility criterion}
The results of the previous subsection make it possible to prove a
sufficient condition for irreducibility of $\V(R)$ when the
$\spd$-action on $R$ is nontrivial. We first need the following lemma.

\begin{lemma}\lbb{cryptic}
Assume the $\Kd$-tensor module $\V(R)$ contains a nonzero
proper submodule. Then the $\spd$-action on $R$ satisfies
\begin{equation}\lbb{identity}
\sum f^{ab}f^{cd}(u) = 0 \,, \qquad u \in R \,,
\end{equation}
for all\/ $1 \leq a,b,c,d \leq 2N$,
where the sum is over all permutations of\/ $a,b,c,d$.
\end{lemma}
\begin{proof}
Let $M$ be a nonzero proper $\Kd$-submodule of $\V(R)$ and
$v \in M$ be an element equal to $\psi(u)$ modulo $\fil^1 \V(R)$
(see \coref{cpsiu}). Let us express $e * v$ in the form
$\sum_I (1 \tt \d^{(I)}) \tt_H u_I$ using \eqref{ksing1} and \eqref{didiv}.
If $|I|>4$ then $u_I = 0$; moreover if $|I| = 4$ then $u_I$ lies
in $\kk \tt R$. By \leref{noconstants} these coefficients must
cancel with each other, and they give exactly \eqref{identity}.
\end{proof}

Now we can prove the main result of this section.

\begin{theorem}\lbb{clarified}
If the $\Kd$-tensor module $\V(\Pi, U, c)$ is not irreducible,
then $U$ is either the trivial representation of\/ $\spd$ or
is isomorphic to $R(\pi_i)$ for some $i=1, \dots, N$.
\end{theorem}
\begin{proof}
Lowering indices in \eqref{identity} gives the following equivalent
identity:
\begin{equation*}
(f_a^b f_c^d + f_{ac} f^{bd} + f_a^d f_c^b + f_c^b f_a^d + f^{bd}
f_{ac} + f_c^d f_a^b) \cdot u = 0,
\end{equation*}
for all $1 \leq a,b,c,d \leq 2N$ and  $u \in R$.
Specializing to $a=b=c=d=i$ we obtain:
\begin{equation*}
4 f_i^i f_i^i + f_{ii} f^{ii} + f^{ii} f_{ii} = 0 \,.
\end{equation*}
Recalling that elements \eqref{fij3} form a standard $\sl_2$-triple,
this can be rewritten as
\begin{equation*}
h_i^2 - e_if_i - f_ie_i = 0 \,.
\end{equation*}
As $h_i^2$ is a linear combination of $h_i^2 - (e_if_i + f_ie_i) =
0$ and the Casimir element of $\langle e_i, h_i, f_i\rangle \simeq
\sl_2$, it acts on any irreducible $\sl_2$-submodule $W \subset U$
as a scalar, which forces $h_i^2$ to be equal either $0$ or $1$. Then
the eigenvalue of the action of $h_i= -2 f_i^i$ on a highest weight
vector in $U$ is also either $0$ or $1$.

When the basis $\d_1, \dots, \d_{2N}$ of $\db$ is symplectic
with respect to $\omega$, elements $h_i = e^i_i - e^{N+i}_{N+i}$
form a basis of the diagonal Cartan subalgebra of $\spd$ (see \eqref{fij4}).
Let $U=R(\lambda)$ be an irreducible representation of $\spd$ with
highest weight $\lambda = \sum_i \la_i \pi_i$. For the
standard choice of simple roots, the eigenvalue of $h_1 = e^1_1 -
e^{N+1}_{N+1}$ on the highest weight vector $\la$ is $\sum_i \la_i \leq1$
(cf.\ \cite[Lecture 16]{FH}). Since $\la_i$ are non-negative
integers, $\lambda$ must be $0$ or one of the fundamental weights $\pi_i$.
\end{proof}

\begin{remark}\lbb{infrank}
The module $\V(\Pi,U,c)$ is always irreducible if the $\spd$-module 
$U$ is infinite-dimensional irreducible. In order to show this, it 
suffices to prove that the factor of $\ue(\spd)$ by 
the ideal generated by relations \eqref{identity} is finite-dimensional. 
It is enough to prove this for the associated graded algebra $S(\spd)$:
letting $a=b=c=d$ in \eqref{identity}, we get $(f^{aa})^2=0$; 
then letting in \eqref{identity} $a=b$, $c=d$, we get
$(f^{ab})^2=-4f^{aa}f^{bb}$, hence $(f^{ab})^4=0$ for all $a,b$, 
proving the claim (we are grateful to C.~De Concini for this argument).
Similarly, the tensor modules for Lie pseudoalgebras of $W$ and $S$ types
in \cite{BDK1} are irreducible if the corresponding modules $U$ 
are irreducible infinite-dimensional.
\end{remark}

We are left with investigating irreducibility of all tensor modules
for which $U$ is isomorphic to some $R(\pi_i)$ or to the trivial
representation $R(\pi_0)=\kk$. We will do so by explicitly constructing all
singular vectors contained in nonzero proper submodules
of $\V(\Pi, U, c)$, and thus
determining conditions on the scalar value $c$ of the action of $I'$.
A central tool for the classification of singular vectors is the
following proposition, which enables us to bound the degree of
singular vectors.

\begin{proposition}\lbb{pboundeddegree}
Let\/ $v \in \V(R)$ be a singular vector contained in a nonzero
proper $\Kd$-submodule $M$, and assume that the\/ $\spd$-action on $R$
is nontrivial. Then $v$ is of degree at most two in the contact filtration, i.e., it is
of the form
\begin{equation}
v = \sum_{i,j =1}^{2N} \d_i \d_j \tt v_{ij} + \sum_{k=0}^{2N} \d_k \tt v_k + 1
\tt \ti v \,.
\end{equation}
\end{proposition}
\begin{proof}
Write $v=\sum_I \d^{(I)} \tt v_I$. Then \leref{lcoefficients},
together with \eqref{ksing3}, shows that $\psi(v_I) =0$ whenever
$|I|'\geq 2$. As the $\spd$-action on $R$ is nontrivial,
$\psi(v_I) = 0$ implies $v_I=0$.
\end{proof}

Our next goal is to characterize singular vectors of degree at most
two in all modules that do not satisfy the irreducibility criterion
given in \thref{clarified}, and thus to obtain a classification of
reducible tensor modules.

\section{Computation of Singular Vectors}\lbb{sksing}

In this section, we will be concerned with tensor modules of the
form $\V(\Pi, U, c)$, where $\Pi$ is an irreducible finite-dimensional
representation of $\dd$, and $U$ is either the trivial
$\spd$-module or one the fundamental representations. Our final
result states that such a tensor module contains singular vectors if
and only if it shows up in a twist of the contact pseudo de Rham
complex, and that in such cases singular vectors may be described in
terms of the differentials.

\subsection{Singular vectors in $\V(\Pi, \kk, c)$}

Here we treat separately the case $U \simeq \kk$.
Since the $\spd$-action is trivial, now \eqref{ksing1} can be rewritten as
\begin{equation}\lbb{spdtrivial1}
\begin{split}
e * (1 & \tt u) = \sum_{k =1}^{2N} (\d_k \tt 1)  \tt_H \bigl( \d^k \tt u -
1 \tt \rho_R(\d^k)u \bigr)\\
&+ (\d_0 \tt 1)  \tt_H (1 \tt cu/2)
- (1 \tt 1)  \tt_H \bigl( \d_0 \tt u - 1 \tt \rho_R(\d_0)u \bigr) \,.
\end{split}
\end{equation}
Using \eqref{p24}, we can also write
\begin{equation}\lbb{spdtrivial2}
\begin{split}
e * (1 \tt u) = &- \sum_{k =1}^{2N} (1 \tt \d_k)  \tt_H \bigl( \d^k \tt u
- 1 \tt \rho_R(\d^k)u \bigr) \\
&- (1 \tt \d_0) \tt_H (1 \tt cu/2)
+ \mbox{ terms in } (\kk \tt \kk) \tt_H \fil^1 \V(R) \,.
\end{split}
\end{equation}

\begin{proposition}\lbb{pu=k}
We have{\rm:}

\medskip

{\rm(i)}
$\sing \V(\Pi, \kk, c) = \fil^0 \V(\Pi, \kk, c)\,$
for $c \neq 0;$

\medskip

{\rm(ii)} $\sing \V(\Pi, \kk,0) = \fil'^1 \V(\Pi, \kk, 0);$

\medskip

{\rm(iii)} $\V(R) = \V(\Pi, \kk, c)$ is irreducible if and only if\/
$c \neq 0$.
\end{proposition}
\begin{proof}
(i) Let $v = \sum_I \d^{(I)} \tt v_I \in \V(R)$ be a singular
vector, and assume that $v_I \neq 0$ for some $I$ with $|I|>0$. If
$n$ is the maximal value of $|I|$ for such $I$, choose among all $I
= (i_0, i_1,\dots, i_{2N})$ with $|I| = n$ one with largest possible
$i_0$. If we use \eqref{spdtrivial2} to compute $e * v$ and express
the result in the form
\begin{equation}\lbb{pbwform}
\sum_J (1 \tt \d^{(J)}) \tt_H u_J \,,\qquad u_J \in \V(R) \,,
\end{equation}
then the coefficient multiplying $1 \tt \d^{(I + \ep_0)}$ equals $-
(i_0 + 1) c v_I /2$. Since $v$ is singular, this must vanish if $|I|
>0$, and $c\neq 0$ gives a contradiction with $v_I \neq 0$.

(ii) In the same way as in part (i), we show that $\sing \V(\Pi,
\kk, 0) \subset \fil^1 \V(\Pi, \kk, 0)$. Indeed, computing the
coefficient multiplying $1 \tt \d^{(I + \ep_k)}$, we see that
$|I|>1$ implies $v_I = 0$. Now, constant vectors are clearly
singular, and using $u = \d_i \tt v_i$ $(i \neq 0)$ in
\eqref{spdtrivial1} easily shows $u$ to be singular for all choices
of $v_i \in R$.
We are left with showing that $\d_0 \tt v_0$ $(v_0 \neq 0)$ is never a
singular vector. Once again, substituting this in
\eqref{spdtrivial1} and expressing the result as in \eqref{pbwform}
gives nonzero terms multiplying $\d_0\d_k\tt 1$ for all $k\neq 0$.

(iii) If $c \neq 0$, then $\V(\Pi, \kk, c)$ has no nonconstant
singular vectors, hence it is irreducible by \coref{constirr}. As
far as $\V(\Pi, \kk, 0)$ is concerned, direct inspection of
\eqref{spdtrivial1} shows that elements  $\d\tt u - 1 \tt
\rho_R(\d)u$ ($\d \in \dd, u \in R$) generate over $H$ a
proper $\Kd$-submodule of $\V(\Pi, \kk, 0)$.
\end{proof}

\begin{corollary}
We have $\sing \V^\Pi_0 = \fil^0 \V^\Pi_0 + \di \fil^0 \V^\Pi_1.$
\end{corollary}
\begin{proof}
The $(\dd \oplus \cspd)$-submodule $\di \fil^0 \V^\Pi_1 \subset \sing
\V^\Pi_0$ contains nonconstant elements, so it has a nonzero
projection to $\gr'^1 \V^\Pi_1$. However, \coref{grdecomposition}
shows that this is isomorphic to $\Pi \boxtimes \overline \dd$,
whence it is irreducible.
\end{proof}

We will now separately classify singular vectors of degree one and
two in all other cases.

\subsection{Classification of singular vectors of degree one}
Our setting is the following: $V = \V(R) = H \tt R$ is a reducible
$\Kd$ tensor module, $R$ is isomorphic to $\Pi \boxtimes U$ as a
$(\dd \oplus \cspd)$-module, where both $\Pi$ and $U$ are
irreducible, and $U = R(\pi_n)$ as an $\spd$-module
for some $1 \leq n \leq
N$. We are also given a nonzero proper submodule $M \subset V$.
Note that by assumption $U$ is not the trivial $\spd$-representation. We look
for singular vectors of degree one, i.e., of the form
\begin{equation}\lbb{v0s1}
v = \sum_{i=0}^{2N} \d_i \tt v_i + 1 \tt \ti v \,,
\end{equation}
which are contained in $M$.
Note that every such singular vector
is uniquely determined by its degree one part.
Indeed, if $v$ and $v'$ are two such vectors agreeing in degree one, then
$v-v'$ is a singular vector contained in $M\cap (\kk \tt R) = \{0\}$
(see \leref{noconstants}).

\begin{lemma}\lbb{s0is0}
If\/ $v\in M$ is a singular vector written as in \eqref{v0s1},
then $v_0 = 0$.
\end{lemma}
\begin{proof}
Compute $e*v$ using \eqref{ksing1}. Then if we write $e *v =
\sum_I (\d^{(I)} \tt 1) \tt_H v_I$, the coefficient multiplying
$\d_0\d_i\d_j \tt 1$ for $i \leq j$ is $\rho_R(f^{ij})v_0\in
M\cap (\kk \tt R)$, up to a nonzero multiplicative constant. Hence
$\rho_R(f^{ij})v_0 = 0$ for all $i,j$, which implies $v_0 = 0$ as the
$\spd$-action is nontrivial.
\end{proof}

\begin{proposition}\lbb{psingunique}
Let\/ $v, v'$ be nonzero singular vectors of degree one contained in
nonzero proper submodules $M, M'$ of a tensor module $\V(R) =
\V(\Pi, U, c)$, as above. If\/ $v = v' \mod \fil^0
\V(R)$, then $v=v'$.
\end{proposition}
\begin{proof}
By \leref{s0is0} and \coref{grdecomposition}, $I'$ acts on $(\sing
\V(R) \cap \fil^1 \V(R))/\fil^0 \V(R)$ via multiplication by $c+1$,
and it obviously acts on $\fil^0 \V(R)$ via multiplication by $c$.
Then $\sing \V(R) \cap \fil^1 \V(R)$ is isomorphic to the direct sum
of the $c+1$ and $c$ eigenspaces with respect to $I'$.

Any $\ti \K$-submodule of $V$ is in particular stable under the
action of $I'$, so it contains the $I'$-eigenspace components of
all of its singular vectors. However, a nonzero proper
submodule $M$ cannot contain constant singular vectors. Thus,
singular vectors must lie in the $c+1$-eigenspace, and their
constant coefficient part is determined by their degree one part,
independently on the choice of the submodule~$M$.
\end{proof}

So far, we have showed that singular vectors of degree one also have
degree one in the contact filtration, and that those contained in a nonzero submodule
must be homogeneous (i.e. eigenvectors) with respect to the action
of $I'$. Notice that since all constant vectors are singular, a
singular vector of degree one stays singular if we alter or suppress
its constant part.

\begin{lemma}\lbb{cisdetermined}
A nonzero element $v = \sum_{k=1}^{2N} \d_k \tt v_k \in H \tt R$ is a
singular vector in $\V(R) = \V(\Pi, U, c)$ for at most one value of
$c$.
\end{lemma}
\begin{proof}
Compute $e * v = \sum_I (\d^{(I)} \tt 1) \tt_H v_I$ using
\eqref{ksing1}. For $k \neq 0$, the coefficient multiplying $\d_0
\d_k \tt 1$ equals $- 1/2 \tt c v_k$ plus a linear combination of
terms of the form $1 \tt \rho_R(f^{ij})v_k$ that arise from
reordering terms multiplying $\d_i\d_j\d_k \tt 1$; such terms are
however independent of the choice of $c$. All such coefficients must vanish when $v$ is
singular. If this happens for two distinct values of $c$, we obtain
$v_k=0$ for all $k$, a contradiction with $v\ne0$.
\end{proof}

\begin{theorem}\lbb{onlyrumin}
Assume that the action of\/ $\spd$ on $U$ is nontrivial. If\/ $V= \V(\Pi, U,
c)$ contains singular vectors of degree one, then $V = \V^\Pi_p$ for
some $1 \leq p \leq 2N+1, p \neq N+1$. More precisely, $\sing
\V^\Pi_p \cap \fil^1 \V^\Pi_p = \fil^0 \V^\Pi_p + \di \fil^0
\V^\Pi_{p+1}$.
\end{theorem}
\begin{proof}
By \thref{clarified}, $\V(R)= \V(\Pi, U, c)$ is irreducible unless
$U = R(\pi_p)$ for some $1 \leq p \leq N.$ \leref{s0is0} and
\coref{grdecomposition} show that singular vectors of degree one in
a nonzero proper $\Kd$-submodule $M$ project faithfully to a $(\dd
\oplus \cspd)$-submodule of $\gr'^1 \V(R)$ isomorphic to $\Pi
\boxtimes (\db \tt U, c+1)$. We can explicitly decompose $\db \tt U$
as a direct sum of irreducibles using \leref{lspfacts}. One has:
\begin{align*}
\db \otimes R(\pi_1) & \simeq R(2\pi_1) \oplus \kk
\oplus R(\pi_{2}) \,,\\
\db \otimes R(\pi_p) & \simeq R(\pi_p + \pi_1) \oplus R(\pi_{p-1})
\oplus R(\pi_{p+1}),\qquad \mbox{if $1 < p < N$} \,,\\
\db \otimes R(\pi_N) & \simeq R(\pi_N + \pi_1) \oplus R(\pi_{N-1})
\,.
\end{align*}
For all values of $1 \leq p \leq N$, the $\spd$-module $R(\pi_p +
\pi_1)$ satisfies the irreducibility criterion stated in
\thref{clarified}, and its dimension is larger than $\dim
R(\pi_p)$. We can therefore proceed as in \cite[Lemma 7.8]{BDK1} to
conclude that no singular vectors will have a nonzero projection to
this summand.

However, by the construction of the contact pseudo de Rham
complex, the tensor module $\V(\Pi, R(\pi_p), c)$ contains
singular vectors projecting to the summand $R(\pi_{p+1})$ when $c
= p$ and to the summand $R(\pi_{p-1})$ if $c = 2N + 2 - p$.
\leref{cisdetermined} shows now that these are the only values of
$c$ for which there are singular vectors projecting to such
components, whereas \prref{psingunique} implies that those are the only
homogeneous singular vectors.
\end{proof}

\subsection{Classification of singular vectors of degree two}

In all of this section, $\V(R) = \V(\Pi, U, c)$ will be a tensor
module containing a singular vector $v$ of degree two. Due to
\prref{pboundeddegree}, we may assume that
\begin{equation}\lbb{v}
v = \sum_{i,j =1}^{2N} \d_i \d_j \tt v_{ij}
+ \sum_{i=0}^{2N} \d_i \tt v_i + 1 \tt \ti v
\end{equation}
where $v_{ij} = v_{ji}$ for all $i,j$. We already know by
\prref{constiff} that $\V(R)$ is reducible, hence $U = R(\pi_p)$ for
some $p$ by \thref{clarified} and \prref{pu=k}. Our goal is to
describe all possible $v$'s, and show that the only tensor modules
possessing them is $\V(\Pi, R(\pi_N), N)$. Recall the definition of
$\psi(u)$ given by \eqref{psiu}.

\begin{lemma}\lbb{spactiondeg2}
We have
$f^{\alpha\beta}(v) = \psi(v_{\alpha\beta}) \mod \fil^1 \V(R)$.
\end{lemma}
\begin{proof}
Use \eqref{ksing2} to compute $e*v$ and compare it with
\eqref{ksing3}.
\end{proof}
This shows that for some $u\in R$ there exists a singular vector
coinciding with $\psi(u)$ modulo $\fil^1 \V(R)$, since if $v$ is a
singular vector of degree two, then $v_{\alpha \beta} \neq 0$ for
some choice of $\alpha,\beta$.
\begin{lemma}
Let $v, v'$ be singular vectors of degree two in $\V(R)$, and
assume that $v=v' \mod \fil^1\V(R)$. Then $v_0=v'_0$.
\end{lemma}
\begin{proof}
Apply \leref{s0is0} to the singular vector of degree one $v-v'$.
\end{proof}
Note that since $I'$ acts on singular vectors, the projection operator $p_2$ of
$\V(R)=\V(\Pi, U, c)$ to the $c+2$ eigenspace with respect to $I'$
maps $\sing \V(R) \cap \fil^2 \V(R)$ to itself. If a nonzero
proper submodule $M$ of $\V(R)$ contains a singular vector $v$ of
degree two, then it also contains $p_2 v$. We will say that $p_2 v$
is a {\em homogeneous} singular vector of degree two.

\begin{lemma}\lbb{phiproperties}
For every $u \in R$ there exists a unique homogeneous singular
vector
\begin{equation}\lbb{phiupsiu}
\phi(u) = \psi(u) \mod \fil^1 \V(R) \,.
\end{equation}
Elements $\phi(u)$ depend linearly on $u$ and satisfy{\rm:}
\begin{align}
\lbb{fabphi}
f^{\alpha\beta}(\phi(u)) & = \phi(f^{\alpha\beta}(u)) \,,\\
\lbb{dphi} \ti\d \cdot \phi(u) & = \phi(\ti\d \cdot u) \,.
\end{align}
Moreover, if $v$ is a homogeneous singular vector of degree two as
in \eqref{v}, then
\begin{equation}\lbb{fabs}
f^{\alpha\beta}(v) = \phi(v_{\alpha\beta}) \,.
\end{equation}
\end{lemma}
\begin{proof}
We know that for some $0 \neq u\in R$ we can find a singular vector
$v$ equal to $\psi(u)$ modulo $\fil^1 \V(R)$. Then its projection $p_2
v$ to the $c+2$-eigenspace of $I'$ is still singular and coincides
with $v$ up to lower degree terms. If we are able to show that
\eqref{fabphi} and \eqref{dphi} hold whenever both sides make sense,
then the set of all $u\in R$ for which $\phi(u)$ is defined is a
nonzero $(\ti\dd \oplus \cspd)$-submodule of $R$, hence all of it by
irreducibility.

So, say $\phi(u)$ is an element as above. By
\leref{spactiondeg2}, we know that $f^{\alpha\beta}(\phi(u))$
coincides with $\psi(f^{\alpha\beta}(u))$ up to lower degree terms.
Moreover, as $I'$ commutes with $\spd$, the vector $f^{\alpha\beta}(\phi(u))$
is still homogeneous, thus showing \eqref{fabphi}. The proof of
\eqref{fabs} is completely analogous. Similarly, \leref{legraction}
implies \eqref{dphi}, as the action of $I'$ commutes with that of
$\ti\dd$.
\end{proof}
\begin{corollary}
The map $\phi\colon R \to \sing \V(R)$ is a well-defined injective
$(\dd\oplus\spd)$-homomorphism, and the action of\/ $\spd$ maps $p_2
\sing \V(R)$ to the image of\/ $\phi$.
\end{corollary}
\begin{proof}
Since we are assuming that the action of $\spd$ on $R$ is
nontrivial, the map $\psi\colon R \to \V(R)$ is injective.
Then, by \eqref{phiupsiu}, $\phi$ is also injective.
\end{proof}

\begin{corollary}
The space $p_2 \sing \V(R)$ does not contain trivial\/
$\spd$-summands.
\end{corollary}
\begin{proof}
If $v\in p_2\sing \V(R)$ lies in a trivial summand, then $0 =
f^{\alpha\beta}(v) = \phi(v_{\alpha\beta})$ for all $\alpha, \beta$.
But $\phi$ is injective, hence $v_{\alpha\beta} = 0$ for all
$\alpha, \beta$, a contradiction with $v$ being of degree two.
Therefore $v=0$.
\end{proof}

The above results can be summarized as follows.

\begin{theorem}
The map $\phi\colon R \to p_2 \sing \V(R)$ is an isomorphism of\/
$(\dd\oplus\spd)$-modules. The action of\/ $I'$ on $p_2 \sing \V(R)$
is via scalar multiplication by $c+2$. All homogeneous singular
vectors of degree two in $\V(R)$ are of the form $\phi(u)$ for $u \in R$.
\end{theorem}

A classification of singular vectors of degree two will now follow
by computing the action of $e\in \Kd$ on vectors of the form
$\phi(u)$. In the computations we will need some identities,
which hold in any associative algebra. We will denote by
$[x, y] = xy-yx$ the usual commutator and by
$$
\{x_1, \dots, x_n\} = \frac{1}{n!} \sum_{\sigma\in S_n} x_{\sigma(1)}
\dots x_{\sigma(n)}
$$
the complete symmetrization of the product.

\begin{lemma}\lbb{formulas}
For any elements $a,b,c,d$ in an associative algebra, we have{\rm:}
\begin{align*}
abc &= \{a, b, c\}
         + \frac{1}{2} \Bigl( \{a, [b, c]\} + \{b, [a, c]\} + \{c, [a, b]\} \Bigr)
         + \frac{1}{6} \Bigl( [a, [b, c]] + [[a, b], c] \Bigr),
\\
{}
\\
abcd & = \{a, b, c, d\}\\
& + \frac{1}{2} \Bigl( \{a, b, [c, d]\} + \{a, c, [b, d]\} + \{a, d, [b, c]\}\\
& \qquad \qquad + \{b, c, [a, d]\} + \{b, d, [a, c]\} + \{c, d, [a, b]\} \Bigr)\\
& + \frac{1}{4} \Bigl( \{[a, b], [c, d]\} + \{[a, c], [b, d]\} + \{[a, d], [b, c]\} \Bigr)\\
& + \frac{1}{6} \Bigl( \{a, [b, [c, d]]\} + \{a, [[b, c], d]\} + \{b,
[a, [c,
d]]\} + \{b, [[a, c], d]\}\\
& \qquad \qquad + \{c, [a, [b, d]]\} + \{c, [[a, b], d]\}
+ \{d, [a, [b, c]]\} + \{d, [[a, b], c]\} \Bigr)\\
& + \frac{1}{6} \Bigl( [[[c, d], b], a] - [[[b, d], c], a] \Bigr)\\
& + \frac{1}{12} \Bigl( [[a, b], [c, d]] + [[a, c], [b, d]] + [[a, d],
[b, c]] \Bigr) \,.
\end{align*}
\end{lemma}
\begin{proof}
It is a lengthy but standard computation. The authors have checked
it using Maple.
\end{proof}

Now let us write
\begin{equation}
\phi(u) = \psi(u) + \sum_{k=0}^{2N} \d_k \tt v_k + 1 \tt \ti v \,,
\qquad u\in R \,,
\end{equation}
for some $v_k,\ti v\in R$, which may depend on $u$.

\begin{lemma}\lbb{s0andC}
If the above vector $\phi(u)$ is singular, then $v_0 = (c/2-N-1)u$.
\end{lemma}
\begin{proof}
We use \eqref{ksing1} to compute $e*\phi(u)=\sum_I (\d^{(I)} \tt 1)
\tt_H v_I$. If $0<a<b$, then the coefficient multiplying
$\d_0\d_a\d_b\tt 1$ equals $I'\cdot f^{ab}(u) - 2f^{ab}(v_0) + $
commutators that are obtained from reordering terms of the form
$\d_i\d_j\d_k\d_l$ in the associative algebra $H = \ue(\dd)$. These
can be computed using \leref{formulas}, leading to
\begin{equation*}
-2f^{ab}(v_0) + I' \cdot f^{ab}(u) + 2\sum_{i,j=1}^{2N}
\om_{ij}[f^{ia}, f^{jb}](u) = f^{ab} \bigl( I' \cdot u - 2v_0 -
(2N+2)u \bigr) \,,
\end{equation*}
where $\sum_{ij} \omega_{ij} [f^{ai}, f^{bj}] = -(N+1) f^{ab}$
due to \eqref{spbracket}.

Since $\phi(u)$ is a singular vector, this coefficient must vanish
for all $a<b$, and a similar computation can be done when $a=b$. Since the
$\spd$-action on $R$ is nontrivial, we obtain that $I' \cdot u - 2v_0 -
(2N+2)u = 0$. To finish the proof, observe that $I' \cdot u = cu$ for all
$u\in R$.
\end{proof}

\begin{lemma}\lbb{s0andcasimir}
If\/ $\phi(u)$ is singular, then
\begin{equation}
c v_0 = \sum_{a,b=1}^{2N} f_{ab}f^{ab} (u) \,.
\end{equation}
\end{lemma}
\begin{proof}
Compute the coefficient multiplying $\d_0^2 \tt 1$ as in
\leref{s0andC}, using \leref{formulas} in order to explicitly
compute commutators arising from terms $\d_0\d_i\d_j$, which cancel,
 and $\d_i\d_j\d_k\d_l$. The final result is
\begin{equation*}
-\frac{1}{2}I' \cdot v_0 + \frac{1}{2} \sum_{i,j,k,l =1}^{2N}
\om_{ik}\om_{jl} f^{ij} f^{kl} (u) = -\frac{1}{2} I' \cdot v_0 +
\frac{1}{2} \sum_{k,l} f_{kl} f^{kl} (u) \,,
\end{equation*}
which is a constant element, and must therefore vanish.
\end{proof}

\begin{corollary}\lbb{cvalues}
If\/ $\phi(u)$ is singular for $0 \neq u\in R= \Pi \boxtimes
(R(\pi_p), c)$, then $c$ equals either $2N + 2- p$ or $p$. In other
words, the only tensor modules that may possess singular vectors of
degree $2$ are of the form $\V^\Pi_p$ or $\V^\Pi_{2N+2-p}$.
\end{corollary}

\begin{proof}
Substitute \leref{s0andC} into \leref{s0andcasimir}, to obtain
\begin{equation*}
\frac{1}{2}c^2 u - (N+1) c u - \sum_{a,b=1}^{2N} f_{ab}f^{ab} (u) = 0
\,.
\end{equation*}
Recall by Lemmas \ref{lsl2} and \ref{lspfacts} that $-\sum_{a,b =1}^{2N} f_{ab}f^{ab}$
equals the Casimir element of $\spd$ and acts on $R(\pi_p)$ via
multiplication by $p(2N+2-p)/2$. Hence we obtain
$$c^2 - (2N+2) c + p(2N+2-p) = 0,$$
whose only solutions are $c = p$ and $c=2N+2-p$.
\end{proof}

\begin{corollary}\lbb{cored}
Let\/ $U$ be a nontrivial irreducible $\spd$-module. Then a tensor
module $V = \V(\Pi, U, c)$ is reducible if and only if it is of the
form $\V^\Pi_p$ for some $1 \leq p \leq 2N+1,$ $p \neq N+1$.
\end{corollary}
\begin{proof}
The image of differentials constitute proper submodules of each
tensor module showing up in the contact pseudo de Rham complex
\eqref{tensorcomplex}.
Conversely, \thref{onlyrumin} and \coref{cvalues} show that there are
no other tensor modules possessing nonconstant singular
vectors.
\end{proof}
\begin{theorem}\lbb{tclassdegreetwo}
The only tensor modules over $\Kd$ possessing singular vectors of
degree two are those of the form $\V^\Pi_N$.
\end{theorem}
\begin{proof}
If $\V(R)=\V(\Pi, R(\pi_p), c)$ has singular vectors of degree
two, then we have a nonzero homomorphism
\begin{equation*}
\V(\Pi, R(\pi_p), c+2) \to \V(\Pi, R(\pi_p), c) \,.
\end{equation*}
However, if $\V(\Pi, R(\pi_p), c+2)$ is irreducible, then this map
is injective, and its image has the same rank as $\V(R)$. Hence, it
is a proper cotorsion submodule $M \simeq \V(\Pi, R(\pi_p), c+2)$ in
$\V(R)$, and the action of $\Kd$ on $\V(R)/M$ is trivial by \reref{rlv0}.
This means that $e *\V(R) \subset (H\tt H)\tt_H M$.
But a direct inspection of \eqref{ksing1} shows that 
$e *\V(R) = (H\tt H)\tt_H \V(R)$, which is a contradiction.

We conclude that $\V(\Pi, R(\pi_p), c+2)$ and $\V(\Pi, R(\pi_p), c)$
are both reducible. By \coref{cored} and \eqref{notation1},
$c$ and $c+2$ must add up to $2N+2$.
Hence, $c=p=N$ and $\V(R) = \V(\Pi, R(\pi_N), N) = \V^\Pi_N$.
\end{proof}

\section{Classification of Irreducible Finite $\Kd$-Modules}\lbb{skmod}

We already know that all $\Kd$-modules belonging to the exact
complex \eqref{tensorcomplex} are reducible, as the image of each
differential provides a nonzero submodule. Further, \coref{cored}
shows that these are the only reducible tensor modules $\V(R)$, when
$R$ is an irreducible finite-dimensional representation of $\dd
\oplus \cspd$.
However, by \prref{pksing1} and \thref{irrfacttens}, every finite
irreducible $\Kd$-module is a quotient of some $\V(R)$, where $R$ is
an irreducible finite-dimensional $(\dd \oplus \cspd)$-module. Thus,
classifying irreducible quotients of all (reducible) tensor modules
$\V(R)$ will yield a classification of all irreducible finite $\Kd$-modules.

\begin{remark}\lbb{singunique}
By Theorems \ref{onlyrumin} and \ref{tclassdegreetwo}, each of the
reducible tensor modules from \eqref{tensorcomplex} contains a
unique irreducible $(\dd \oplus \cspd)$-summand of nonconstant
singular vectors.
\end{remark}

\begin{lemma}\lbb{Hlin}
Let\/ $V$ and\/ $W$ be $\Kd$-modules, and assume $V$ is generated
over $H$ by its singular vectors. If\/ $f\colon V \to W$ is a
$\Kd$-homomorphism,
then $f(V)$ is also $H$-linearly generated by its singular vectors.
\end{lemma}
\begin{proof}
This follows immediately from $f(\sing V) \subset \sing W$.
\end{proof}

\begin{theorem}\lbb{thm9.1}
The image modules $\diru \V^\Pi_{N+2}$ and\/ $\di \V^\Pi_{p+1}$, where $0 \leq
p \leq 2N+1$, $p \neq N, N+1$ are the unique nonzero proper
$\Kd$-submodules
of\/ $\V^\Pi_N$ and\/ $\V^\Pi_p$, respectively.
\end{theorem}
\begin{proof}
We first claim that these submodules are irreducible, hence minimal.
By \prref{pksing1} and \reref{singunique}, it is enough to show that
they are $H$-linearly generated by their singular
vectors. This follows from \leref{Hlin}.

To prove that there are no other nonzero proper submodules,
it is enough to show that these minimal submodules are also
maximal. Equivalently, the quotients $\V^\Pi_N / \diru \V^\Pi_{N+2}$
and $\V^\Pi_p / \di \V^\Pi_{p+1}$ are irreducible, which follows from
exactness of the complex \eqref{tensorcomplex}.
\end{proof}

The above results lead to the main theorem of the paper.

\begin{theorem}\lbb{thm9.2}
A complete list of non-isomorphic finite irreducible $\Kd$-modules is
as follows{\rm:}

\medskip
{\rm(i)}
Tensor modules\/ $\V(\Pi, U)$ where $\Pi$ is an irreducible
finite-dimensional representation of\/ $\dd$ and\/
$U$ is a nontrivial irreducible finite-dimensional\/ $\cspd$-module\/
not isomorphic to $(R(\pi_p), p)$ or $(R(\pi_p), 2N+2-p)$ with\/
$1 \leq p \leq N,$

\medskip
{\rm(ii)}
Images of differentials in the twisted contact pseudo de Rham complex
\eqref{tensorcomplex}, namely\/
$\diru \V^\Pi_{N+2}$ and\/ $\di\V^\Pi_n$ where
$1 \leq n \leq 2N+1$, $n \neq N+1,N+2$.
Here\/ $\Pi$ is again an irreducible finite-dimensional\/ $\dd$-module.
\end{theorem}

\begin{remark}
The image $\di\V^\Pi_{2N+2}$
of the first member of the complex \eqref{tensorcomplex}
is isomorphic to $\V^\Pi_{2N+2} = \V(\Pi,R(\pi_0), 2N+2)$
and it is included in part {\rm(i)} of the above theorem.
\end{remark}

Recall that representations of the Lie \psalg\ $\Kd$ are
in one-to-one correspondence with conformal representations of the extended
annihilation algebra $\ti\K$ (see \cite{BDK} and \prref{preplal2}).
The latter is a direct sum of Lie algebras $\ti\K=\ti\dd\oplus\K$
where $\ti\dd\simeq\dd$ and $\K$ is isomorphic to the Lie--Cartan algebra
$K_{2N+1}$ (see Propositions \ref{plck} and \ref{pknorm}).
Thus, from our classification of finite irreducible $\Kd$-modules
we can deduce a classification of irreducible conformal $K_{2N+1}$-modules.
In this way we recover the results of I.A.~Kostrikin, which were stated
in \cite{Ko} without proof.

In order to state the results, we first need to set up some notation.
Let $R$ be a finite-dimensional representation of $\cspd$.
Using that $\cspd\simeq \K'_0/\K'_1$, we endow $R$ with an
action of $\K'_0$ such that $\K'_1$ acts trivially
(see \prref{pk0k1}). We also view $R$ as a
($\dd\oplus\cspd$)-module with a trivial action of $\dd$, and as before
we write $R=(U,c)$. Then, by \reref{rtm2} and \prref{pknorm},
the induced $\K$-module $\Ind_{\K'_0}^{\K} R$ is isomorphic to the tensor
$\Kd$-module $\V(R)=\V(\kk,U,c)$. Finally, let us recall the $\ZZ$-grading
of $V$ introduced in the proof of \prref{constiff}.
In this setting, Theorems \ref{thm9.1} and \ref{thm9.2} along with \reref{infrank} imply the following.

\begin{corollary} {\rm(i) \cite{Ko}.}
Every nonconstant homogeneous singular vector in\/ $V=\V(\kk,U,c)$
has degree\/ $1$ or\/ $2$. The space\/ $S$ of such singular vectors is an\/
$\spd$-module, and the quotient of\/ $V$ by the\/ $\K$-submodule generated
by\/ $S$ is an irreducible\/ $\K$-module.
All singular vectors of degree\/ $1$ in\/ $V$ are listed in cases
{\rm(a), (b)} below, while all singular vectors of degree\/ $2$ are listed
in\/ {\rm(c):}
\begin{align}
\tag{a}
U&=R(\pi_p) \,, \quad c=p \,, \quad S=R(\pi_{p+1}) \,,
\qquad 0\leq p\leq N-1 \,.
\\
\tag{b}
U&=R(\pi_p) \,, \quad c=2N+2-p \,, \quad S=R(\pi_{p-1}) \,,
\qquad 1\leq p\leq N \,.
\\
\tag{c}
U&=R(\pi_N) \,, \quad c=N \,, \quad S=R(\pi_{N}) \,.
\end{align}

{\rm(ii) \cite{Ko}.}
If the $\spd$-module $U$ is infinite-dimensional irreducible,
then $\V(\kk,U,c)$ does not contain nonconstant singular vectors.

{\rm(iii)}
If a\/ $\K$-module\/ $\V(\kk,U,c)$ is not irreducible, then its 
$($unique$)$ irreducible quotient is isomorphic to the topological dual of the kernel
of the differential of a member of the Rumin complex over formal power 
series.
\end{corollary}

\bibliographystyle{amsalpha}


\end{document}